\documentclass[12pt,reqno, twoside,english]{amsart}
\usepackage[T1]{fontenc}
\usepackage[utf8]{inputenc}
\usepackage{geometry}
\usepackage[english]{babel}
\usepackage{enumitem}
\usepackage{verbatim}
\usepackage{mathrsfs}
\usepackage{amsbsy}
\usepackage{amstext}
\usepackage{amsthm}
\usepackage{amssymb}
\usepackage[all]{xy}
\usepackage{mathrsfs}
\usepackage[normalem]{ulem}
\usepackage{url}            
\usepackage{booktabs}       
\usepackage{amsfonts}       
\usepackage{nicefrac}       
\usepackage{microtype}      
\usepackage[pdftex]{graphicx}
\usepackage{todonotes}
\usepackage{pdfpages}
\usepackage{amsmath}
\usepackage{tikz-cd}
\usepackage{amsfonts}
\usepackage{mathrsfs}  
\usepackage{centernot}
\usepackage{amssymb}
\usepackage{amsthm}
\usepackage{hyperref}
\usepackage{enumitem}
\usepackage[noadjust]{cite}
\usepackage{caption}
\usepackage{dutchcal}
\usepackage{bbm}
\usepackage{tikz}
\usetikzlibrary{decorations.pathreplacing}

\newcommand{\df}{{\, \stackrel{\mathrm{def}}{=}\, }}
\usepackage{mathtools}
\usepackage{stackengine}
\usepackage{amsmath}
\usepackage{tikz}
\usepackage{tikz-cd}
\usepackage{bm}
\usepackage{mathrsfs}
\usepackage{booktabs}
\usetikzlibrary{shapes.geometric}
\usepackage{amscd}
\usepackage{mathabx}
\numberwithin{equation}{section}
\usepackage[original]{imakeidx}
\makeindex 
\usepackage{mathrsfs}
\usepackage{graphics}
\usepackage{eucal}
\usepackage{mathrsfs}
\usepackage{mdwlist}
\usepackage{color}
\usepackage{cleveref}
\crefformat{section}{\S#2#1#3}
\crefformat{subsection}{\S#2#1#3}
\crefformat{subsubsection}{\S#2#1#3}

\makeatletter
\@namedef{subjclassname@2020}{%
  \textup{2020} MSC}
\makeatother

\hypersetup{
	colorlinks   = true, 
	urlcolor     = blue, 
	linkcolor    = purple, 
	citecolor   = blue 
}

\def\XXint#1#2#3{{\setbox0=\hbox{$#1{#2#3}{\int}$ }
\vcenter{\hbox{$#2#3$ }}\kern-.6\wd0}}

\makeatletter
\newcommand*{\rom}[1]{\expandafter\@slowromancap\romannumeral #1@}

\newtheorem{theorem}{Theorem}[section]
\newtheorem{remark}[theorem]{Remark}

\newtheorem{proposition}{Proposition}[section]
\newtheorem{question}[theorem]{Question}

\newtheorem{lemma}[theorem]{Lemma}

\newtheorem{definition}[theorem]{Definition}

\theoremstyle{definition}

\newcommand{\x}{\mathbf{x}}

\newcommand{\p}{{\mathbf{P}}}

\newcommand{\Z}{\mathbb{Z}}
\newcommand{\R}{\mathbb{R}}

\newcommand{\N}{\mathbb{N}}

\newcommand{\y}{\mathbf{y}}

\newcommand{\K}{\mathcal{K}}

\newcommand{\F}{\mathbb{F}}

\newcommand{\details}[1]{}

\newcommand{\beq} {\begin{equation}}
\newcommand{\eeq} {\end{equation}}

\usepackage{fancyhdr}
\DeclareMathOperator{\codim}{\text{codim}}
\begin{document}

\title{Simultaneous Khintchine theorem on manifolds in positive characteristics: convergence case }

\author{Noy Soffer Aranov}
\address{\textbf{Noy Soffer Aranov} Graz University of Technology, Institute of Analysis and Number Theory, 8010 Graz, Austria}
\email{noy.sofferaranov@tugraz.at}
\author{Sourav Das}
\address{\textbf{Sourav Das} School of Mathematics, TIFR Mumbai}
\email{sourav@math.tifr.res.in, iamsouravdas1@gmail.com}
\author{Arijit Ganguly}
\address{\textbf{Arijit Ganguly} Department of Mathematics, IIT Kanpur}
\email{arijit.ganguly1@gmail.com, aganguly@iitk.ac.in}
\author{Aratrika Pandey}
\address{\textbf{Aratrika Pandey} Department of Mathematics, IIT Bombay, Mumbai, India 400076 \\}
\email{aratrika.pandey@gmail.com, 214090004@iitb.ac.in}
\thanks{}
\begin{abstract}
We prove the convergence case of Khintchine’s theorem, with general approximation functions $\psi$ that are not necessarily monotonic, for analytic nonplanar manifolds over local fields of positive characteristic. Our approach is based on the method of counting rational points near manifolds developed by Beresnevich and Yang \cite{BY}. To address the scenario where the given $\psi$ is not monotonic, we extend our function field by adjoining an appropriate root. Additionally, in the course of the proof, we establish several new results in the geometry of numbers over function fields, which are of independent interest.
\end{abstract}
\maketitle

\section{Introduction}
Diophantine approximation is a central topic in number theory, which deals with the effective density of rational numbers within real numbers and its higher-dimensional analogue. The Khintchine-Groshev theorem is a foundational result in metric Diophantine approximation. Let $\psi:(0,\infty) \longrightarrow (0,1)$ be a \emph{approximating function}, i.e., $\psi(x)\rightarrow 0$ as $x \rightarrow \infty$. For a fixed $\boldsymbol{\theta}\in \mathbb R^n$ with $n \in \mathbb N$, the following set has been an object of interest for a long time
$$
    \mathcal{S}_{n}^{\boldsymbol{\theta}}(\psi)\df\left\{\y\in  \mathbb R^n: \left\|\y-\frac{\mathbf{p}+\boldsymbol{\theta}}{q} \right\|<\frac{\psi(q)}{q}\,\,\text{for i.m. }(\mathbf{p},q) \in \mathbb Z^n \times \mathbb N \right\},
$$
where $\|\cdot\|$ denotes the sup norm and `i.m' reads as `infinitely many'. The points $\y$ lying in $\mathcal{S}_{n}^{\boldsymbol{\theta}}(\psi)$ are usually referred to as \emph{$(\psi,\boldsymbol{\theta})$-approximable}. 
The case $\boldsymbol{\theta} = \mathbf{0}$ is called \emph{the homogeneous setting}, and the points of $\mathcal{S}_{n}^{\mathbf{0}}(\psi)$ are referred \emph{to as $\psi$-approximable}. The form of approximation described above concerns how well one can approximate points of $\mathbb{R}^n$ by rational points, and it is commonly referred to as the \emph{simultaneous form} of approximation. 
In contrast, another widely studied type of approximation asks how close a point in $\mathbb{R}^n$ is to a rational hyperplane; this is known as the \emph{dual form} of approximation. To begin with, we recall the inhomogeneous version of classical Khintchine's theorem (see \cite{Khinchine,Groshev} and \cite[\S 1.1]{AllenRamirez}).

\begin{theorem}[The Inhomogeneous Khintchine theorem]
\label{thm:InhomKhint}
Given any approximating function $\psi$ and $\boldsymbol{\theta} \in \mathbb R^n$
\begin{equation}\label{inhomokhint}
\mathcal{S}_n^{\boldsymbol{\theta}}(\psi)=\begin{cases}
\text{Lebesgue null},  & \text{if } \displaystyle \sum_{q=1}^{\infty}\psi(q)^n<\infty;\\
\text{Lebesgue full},   & \text{if }\displaystyle \sum_{ q=1}^{\infty}\psi(q)^n=\infty\text{ and }\psi \text{ is monotonic}.
\end{cases}\end{equation} 
\end{theorem}
It is important to note that, while one does not need any sort of monotonicity assumption to prove the convergence part, the question of whether the emphasized monotonicity assumption can be dropped in the divergence case has attracted many mathematicians over the last six or seven decades at least. For $n=1$, this assumption is crucial, owing to the counterexample constructed by Duffin and Schaeffer \cite{DuffinSchaeffer}. In fact, it has been shown in \cite{Ram} that, when $n=1$, for no $\theta \in \mathbb{R}$ one can remove the monotonicity.  In a recent paper \cite{Yu}, Han Yu showed that the monotonicity assumption can be removed for all $n\geq 3$. In the case of $n=2$, Allen and Ram\'{i}rez were able to remove the same at the cost of an extra divergence of a measure sum \cite[Theorem 3]{AllenRamirez}. \\

The situation becomes more delicate when examining the extent to which embedded submanifolds of 
$\R^n$ inherit the Diophantine properties prevalent in $\R^n$. Establishing an analogue of Khintchine's theorem for manifolds is closely linked to the challenging problem of counting rational points near the manifold (for example, see \cite{Ber,AnnalsBeres,BVVZ,SST,beresnevich2025rational}). For analytic non-degenerate manifolds of dimension 
$d\geq2$, the divergence case was established by Beresnevich \cite{AnnalsBeres}, and later extended to non-analytic curves in \cite{BVVZ}, both employing the ubiquity framework alongside Kleinbock–Margulis's quantitative non-divergence estimates \cite{KM1998}. For non-degenerate manifolds in $\R^n$, the convergence case was resolved by Beresnevich and Yang \cite{BY} (see also \cite{SST}), and Khintchine’s theorem in full generality for simultaneous approximation was recently settled by Beresnevich and Datta \cite{beresnevich2025rational}. More recently, Datta \cite{Dattaquant} extended the result in Beresnevich and Yang by proving a quantitative simultaneous Khintchine theorem over manifolds in $\R^n$. We ask the reader to note that all have assumed the approximation function is non-increasing, even for the convergence case.\\

On a related note, we remind the reader that, unlike the simultaneous form, the dual form of Diophantine approximation is well understood. Beginning with Kleinbock and Margulis’s resolution of the Baker-Sprind\v{z}uk conjecture \cite{KM1998}. A complete dual Khintchine theorem for non-degenerate manifolds was later established in both homogeneous \cite{BKM, BBKM} and inhomogeneous settings \cite{BBV}.\\

In recent years, there has been growing interest in Diophantine approximation over local fields of positive characteristic. Diophantine approximation over function fields concerns the quantitative study of approximating Laurent series by rational functions and their higher-dimensional analogues. In this context, Mahler developed the geometry of numbers in the seminal paper \cite{Mahler}, which provides a straightforward route to proving the analogue of Dirichlet's theorem. The Khintchine theorem in positive characteristic was proved by de Mathan \cite{deMathan}, later extended by Kristensen to systems of linear forms and to the analysis of Hausdorff dimensions of exceptional sets \cite{Kr1}. Further advances include a multiplicative Khintchine–Groshev theorem \cite{AGP} and extensions to imaginary quadratic function fields \cite{GR}. More recently, Chao Ma and Wei-Yi Su established inhomogeneous versions of the Khintchine and Jarník–Besicovitch theorems \cite{Chao-Su}, and Kristensen derived asymptotic formulas for inhomogeneous linear systems, yielding inhomogeneous forms of the Khintchine–Groshev and Jarník theorems \cite{Kr2}. For further developments in the inhomogeneous theory, see \cite{KN, Fuchs}.\\

In the metric theory over manifolds in this setting, the analogues of the Baker-Sprind\v{z}uk conjecture had been proved by Ghosh \cite{Ghosh2007}. Subsequently, Ganguly and Ghosh proved the inhomogeneous Sprind\v{z}uk conjecture in \cite{GGcontemp}.
Recently, Das and Ganguly proved a complete inhomogeneous Khintchine-Groshev type theorem for nonplanar manifolds over function fields in the dual setting \cite{das2022inhomogeneous}. For additional results and recent developments on Diophantine approximation in function fields, see \cite{GG1,GJNT,DaG2, Bang, AranovKim, DattaXu, KimLinPaulin}.\\

In this paper, we prove the convergence case of Khintchine’s theorem for analytic non-planar manifolds over local fields of positive characteristic without the monotonicity assumption. While this certainly extends the earlier works of Beresnevich and Yang \cite{BY}, and Beresnevich and Datta \cite{beresnevich2025rational} to the setting of function fields, we believe that the fact that our result applies to general approximation functions that need not be monotonic makes it distinguishable from its real counterpart.  Although our approach is based on the technique of Beresnevich and Yang \cite{BY} for counting rational points near manifolds, the positive characteristic setup and the absence of the monotonicity necessitate proper adaptation, which we explain in \S\ref{subsec:method}. On our way, we also prove some new results in the geometry of numbers over function fields (see \S\ref{sec:geometry of numbers} and \S\ref{sec:minima}), which are of independent interest. We begin by introducing the function field framework.
\subsection{The Function Field Setting}\label{sec:main_result}
We begin with the field of rational functions $\frac{P}{Q} \in \mathbb{F}_q(T)$, where $q = p^b$ for some prime $p$ and $b \in \mathbb N$. We define a non-archimedean absolute value on $\mathbb{F}_q(T)$ as follows.  For any rational function $\frac{P}{Q} \in \mathbb{F}_q(T)$, where $P, Q \in \mathbb{F}_q[T]$ and $Q \neq 0$, we set
\[
\left|\frac{P}{Q}\right| = 
\begin{cases}
q^{\deg(P) - \deg(Q)}, & \text{if } P \neq 0, \\
0, & \text{if } P = 0.
\end{cases}
\]
The completion of $\mathbb{F}_q(T)$ with respect to this absolute value is $\mathbb{F}_q((T^{-1}))$, i.e., the field of Laurent series in $T^{-1}$
over the finite field $\mathbb F_q.$ The absolute value on $\mathbb{F}_q((T^{-1}))$, which we again denote by $|\cdot|$, is defined as follows.  For $a \in \mathbb{F}_q((T^{-1}))$, if $a=0$ then $|a| = 0$. Otherwise, $a$ can be expressed uniquely as a Laurent series,
\[
a = \sum_{k \leq k_0} a_k T^k,
\]
where $k_0 \in \mathbb{Z}$, each $a_k \in \mathbb{F}_q$, and the leading coefficient $a_{k_0} \neq 0$. We define the degree of $a$ by $\deg(a) \df k_0$, and set $|a| \df q^{\deg(a)}$. This clearly extends the absolute value $|\cdot|$ on $\mathbb{F}_q(T)$ to $\mathbb{F}_q((T^{-1}))$, and the extension remains non-archimedean and discrete. Consequently, $\mathbb{F}_q((T^{-1}))$ is a complete and separable metric space that is ultrametric and, hence, totally disconnected. It is worth noting that any local field of positive characteristic is isomorphic to some 
 $\mathbb{F}_q((T^{-1}))$ (\cite[Part I, Chapter I, Theorem 8]{Weil}).\\

Throughout the paper, we let $\mathcal{R}$ and $\K$ denote, respectively, the polynomial ring $\mathbb F_q[T]$ and the field of Laurent series $\mathbb{F}_q((T^{-1})).$
We denote by $\mathcal{O}_{\K}$ the ring of integers of $\K$, defined by $\mathcal{O}_{\K}\df\{x\in \K:|x|\leq 1\}$. Let $\mathcal{L}_1$ denote the Haar measure on $\K$, normalized so that $\mathcal{L}_1\left(\mathcal{O}_{\K}\right)=1$. For any $n \in \N$, we equip $\K^n$ with the sup norm, which is defined as follows: 
\[
\|\x\| = \max_{1 \leq i \leq n} |x_i|,\quad \text{for}\,\, \x=(x_1,\dots,x_n) \in \K^n, 
\]
and we equip it with the product Haar measure, denoted by $\mathcal{L}_n$.\\

The field $\K = \mathbb{F}_q((T^{-1}))$ is not algebraically closed. For every integer $\ell \ge 2$, the equation $
x^\ell = T$ has no solution in $\K$. Consequently, the algebraic closure of $\K$ is an infinite algebraic extension of $\K$. (see \cite{Ked} for more information about the algebraic closure of $\K$). Nevertheless, one can discuss finite extensions, such as the extension of the function field $\K$ obtained by adjoining the $\ell$-th root of $T$, and we denote this extension by $\K_{\ell}$. Let the polynomial ring in $\K_{\ell}$ be denoted by $\widetilde{\mathcal{R}}$, i.e.,   $\widetilde{\mathcal{R}}\df \F_q[T^{1/\ell}]\subseteqq \K_\ell$. So we have the natural inclusion $\mathcal{R} \subseteqq \widetilde{\mathcal{R}} \subseteqq \K_{\ell}.$
It is worth noting that $\widetilde{\mathcal{R}}$ is a lattice in $\K_{\ell}$, while $\mathcal{R}$ is only a discrete subgroup when viewed as a subset of $\K_{\ell}$. We also equip $\K_{\ell}^n$ with an analogous sup norm. \\   

Given an \emph{approximating function} $\psi:\{q^r: r\in \mathbb{Z}_{\geq 0}\}\longrightarrow \{q^r: r\in \mathbb{Z}\}$  and a vector $\boldsymbol{\Theta} \in \K^n$, we define the set of all \emph{$(\psi, \boldsymbol{\Theta})$-approximable vectors} in $\K^n$ by
\[
\mathcal{S}_n^{\boldsymbol{\Theta}}(\psi) \df \left\{ \x \in \K^n : \left\| \x - \frac{\mathbf{P} + \boldsymbol{\Theta}}{Q} \right\| < \frac{\psi(|Q|)}{|Q|} \quad \text{for infinitely many } (\mathbf{P}, Q) \in \mathcal{R}^n \times (\mathcal{R} \setminus \{0\}) \right\}.
\]
In the special case when $\boldsymbol{\Theta} = \mathbf{0}$, the set reduces to the set of \emph{$\psi$-approximable vectors}, which we denote by $\mathcal{S}_n({\psi}) \df \mathcal{S}_n^{\mathbf{0}}(\psi).$\\

Without loss of generality, we consider the manifold $\mathscr{M}$ defined by the map $\mathbf{f}: U\subseteqq \K^d \to \K^n$, where
\[
\mathbf{f} \df (x_1, \ldots, x_d, f_1(\x), \ldots, f_m(\x)) = (\x, \bm{f}(\x)),
\]
with $d = \dim(\mathscr{M})$, $m= \codim(\mathscr{M})$, $n=d+m$, and $U\subseteqq \K^d$ is an open set. We also assume that the defining map $\mathbf{f}$ is analytic (see \cref{sec:function} for the definition).
\subsection{Main Results}
Throughout the paper, we make the following assumptions for the sake of simplicity:

\begin{enumerate}[label=($\mathfrak{A}$.{\arabic*})]
\item \label{I}$\overline{U}$ is compact.
       \item\label{II} The map $\bm{f}=(f_1,\dots,f_m): U \subseteqq \K^d \rightarrow \K^m$ admits an analytic extension to an open neighborhood of $\overline{U}$. 
       
    \item\label{III} The restrictions of the functions $1,f_1,\dots,f_m$ to any open subset of $U$ are linearly independent over $\K.$

    \item\label{IV} $M>0$ is such that 
    \begin{enumerate}[label=(\roman*)]
        \item  $\|\bm{f}(\x)\| \leq M$ and $\|\nabla \bm{f}(\x)\| \leq M$, for any $\x \in U$; and 
        \item \label{second-diff-quo} $|{\bar\Phi_{\beta}}\bm{f}(\y_1,\y_2,\y_3)|\leq M$ for any second difference quotient $\Phi_{\beta}$ (see \cref{sec:function} for the definition) and $\y_1,\y_2,\y_3 \in U$. 
    \end{enumerate}
    \end{enumerate}
\vspace{0.1 cm}
We now state the main result of this article.
\begin{theorem}
\label{thm:main}
Let $n \geq 2$, let $\boldsymbol{\Theta}\in \K^n$, and let $\mathscr{M} \subseteqq \K^n$ be an analytic submanifold satisfying \ref{I}-\ref{IV}. Suppose that the approximating function $\psi:\{q^r: r\in \mathbb{Z}_{\geq 0}\}\longrightarrow \{q^r: r\in \mathbb{Z}\}$ satisfies
\begin{equation}\label{conv-sum}
    \sum_{\tiny {Q \in \mathbb{F}_q[T] \setminus \{0\}}} \psi(|Q|)^n < \infty.
\end{equation}
Then almost every point of $\mathscr{M}$ is not $(\psi,\boldsymbol{\Theta})$-approximable, i.e.,
$\mathcal{L}_d\left(\mathbf{f}^{-1}(\mathcal{S}_n^{\boldsymbol{\Theta}}(\psi))\right)=0$. 
\end{theorem}
\begin{remark}$ \ $
    \begin{enumerate}[label=($\mathfrak{R}$.\arabic*)]
        \item\label{rem-conv sum} Given any $\varphi:\{q^r: r\in \mathbb{Z}_{\geq 0}\}\longrightarrow \{q^r: r\in \mathbb{Z}\}$, it is easy to see that 
\begin{equation}\label{conv-sum alt}
    \sum_{\tiny {Q \in \mathbb{F}_q[T] \setminus \{0\}}} \varphi (|Q|)< \infty \Longleftrightarrow \sum_{t = 0}^{\infty} q^t \varphi(q^t) < \infty.
\end{equation}
In particular, \eqref{conv-sum} is equivalent to 
$\sum_{t = 0}^{\infty} q^t \psi(q^t)^n < \infty$.
Indeed, given any $t\geq 0$, an easy counting argument yields that there are exactly $(q-1)q^t$ polynomials of degree $t$. Therefore, 
\[
\sum_{Q \in \mathbb{F}_q[T],~|Q|\leq q^N} \varphi(|Q|)
=(q-1) \sum_{t=0}^N  q^t \varphi(q^t), \text{ for any } N \in \N. 
\] Note that the real analogue of this equivalence requires that $\psi$ is non-increasing (see \cite[Lemma 1.4]{BY}), whereas here this is not necessary due to \eqref{conv-sum alt}.
\item Fix $\kappa \in \left(1, \frac{3}{2}\right)$. It suffices to prove Theorem \ref{thm:main} for approximating functions satisfying the following:
\begin{equation}\label{lower bound approx} 
    \psi\left(|Q|\right)\geq \frac{1}{|Q|^{\frac{\kappa}{n}}},~\forall Q\in \mathbb{F}_q[T]\setminus \{0\}.
\end{equation}For, if $\psi$ does not satisfy \eqref{lower bound approx}, one considers the function
\[\tilde{\psi}(|Q|)\df\max \left\{\psi(|Q|), \frac{1}{|Q|^{\frac{\kappa}{n}}} \right\},~~\forall Q\in \mathbb{F}_q[T]\setminus \{0\}.\] It is then easy to verify that 
\begin{align*}
    \sum_{Q} \tilde{\psi}(|Q|)^n \leq  \sum_{Q } \psi(|Q|)^n+  \sum_{Q}\frac{1}{|Q|^{\kappa}}\asymp     \sum_{t = 0}^{\infty} q^t \psi(q^t)^n + \sum_{t=0}^{\infty}q^t\cdot \frac{1}{q^{\kappa t}}<\infty.
\end{align*} Since $\tilde{\psi}$ clearly satisfies \eqref{conv-sum} and $(\psi,\boldsymbol{\Theta})$-approximability implies $(\tilde{\psi},\boldsymbol{\Theta})$-approximability, one concludes $\mathcal{L}_d\left(\mathbf{f}^{-1}(\mathcal{S}_n^{\boldsymbol{\Theta}}(\psi))\right)=0$ at once from  $\mathcal{L}_d\left(\mathbf{f}^{-1}(\mathcal{S}_n^{\boldsymbol{\Theta}}(\tilde{\psi}))\right)=0$. In view of this, the use of the assumption mentioned above in \eqref{lower bound approx} in the proof of Theorem \ref{thm:main} is completely justified.  
    \end{enumerate}
\end{remark}
 A simple argument further extends Theorem \ref{thm:main} to  Hausdorff measures. Recall that, for $ \sigma>0$, the Hausdorff $\sigma$-measure of a set $A \subseteqq \K^{n}$ is defined as
\begin{equation*}
    \mathcal{H}^{\sigma}(A)\df\lim_{\rho \to 0^{+}} \inf \left\{ \sum_{i=1}^\infty \operatorname{diam}(A_{i})^{\sigma} : A \subseteqq \bigcup_{i=1}^\infty  A_{i}\text{ and } \operatorname{diam}(A_{i})<\rho,~\forall i \right\},
\end{equation*}
where $\operatorname{diam}(A)=\sup\{\Vert \mathbf{u}-\mathbf{v}\Vert:\mathbf{u},\mathbf{v}\in A\}$ denotes the diameter of a set. The following provides an extension of Theorem~\ref{thm:main} to Hausdorff measures. 
\begin{theorem}
\label{thm:DimConv}
    Let $d+m=n\geq 2$, $\boldsymbol{\Theta}\in \K^n$, and $\psi:\{q^r: r\in \mathbb{Z}_{\geq 0}\}\longrightarrow \{q^r: r\in \mathbb{Z}\}$ be an approximating function. Suppose that $\mathscr{M} \subseteqq \K^n$ is an analytic submanifold satisfying \ref{I}-\ref{IV}. Assume that $\sigma>0$ and $\alpha>0$ are such that the following pair of series is convergent: 
    \begin{equation}\label{eq:hausdim}
      \sum_{Q\in \F_q[T]\setminus \{0\}}\left(\frac{\psi(|Q|)}{|Q|}\right)^{\sigma+m}|Q|^{n} \text{ and }  \sum_{Q\in \F_q[T]\setminus\{0\}}\left(\frac{\psi(|Q|)}{|Q|}\right)^{\frac{\sigma-d}{2}}\left(\psi(|Q|)^{-\frac{n\alpha}{n+1}}|Q|^{-\frac{3\alpha}{2(n+1)}-1}\right).
    \end{equation}
    Then, $\mathcal{H}^\sigma\left(\mathbf{f}^{-1}(\mathcal{S}_n^{\boldsymbol{\Theta}}(\psi))\right)=0$.
\end{theorem}
 We prove both the above theorems in \S \ref{sec: proof of main theorem}. Similar to \cite{BY}, the proof strategy is to divide the manifold into two parts: we count the number of rational points near the manifold in the \emph{generic part} $U \setminus \mathfrak{M}(s,t)$ and show that the measure of the \emph{special part} $\mathfrak{M}(s,t)$ of the manifold is relatively small. We refer the reader to \S \ref{subsec:Generic} for the definition of $\mathfrak{M}(s,t)$. \\

Given integers $t,s > 0$ and $\Delta \subseteqq \K^n$, we define the following set, which counts the number of rational points $\frac{\p}{Q}$ that are within a distance $q^{-(s+t)}$ of $\mathbf{f}(\Delta \cap U)$:
\begin{align*}
\mathcal{R}_{\boldsymbol{\Theta}}(\Delta,s,t) \df& \left\{(\mathbf{P},Q) \in \mathcal{R}^{n}\times (\mathcal{R} \setminus \{0\}) : |Q|=q^t\text{ and } \inf_{ \x \in \Delta \cap U} \left\|\mathbf{f}(\x) - \frac{\p + \boldsymbol{\Theta}}{Q}\right\| < q^{-(s+t)} \right\},
\end{align*}
and  $\mathcal{N}_{\boldsymbol{\Theta}}(\Delta; s,t) \df \# \mathcal{R}_{\boldsymbol{\Theta}}(\Delta,s,t)$. We prove the following proposition.

\begin{proposition}
\label{prop:Key}
Let $U \subseteqq \K^d$ be an open set, and let $\mathbf{f} : U \to \K^n$ be a map satisfying \ref{I}-\ref{IV}, and $1<\kappa <\frac{3}{2} $.  Then there exists a family $\left\{\mathfrak{M}(s,t):s,t\in \N\right\}$ of subsets of $ U$ satisfying the following: 
\begin{enumerate}[label=($\mathfrak{S}$.\arabic*)]
    \item $\mathfrak{M}(s,t)$ can be expressed as the union of disjoint balls in $U$ with radius $q^{-\frac{s+t}{2}}$; 
    \item\label{special part} for every $\x_0\in U$, there exists a ball $B_0$ centered at $\x_0$ and  positive constants $C$ and $\alpha$, depending on $B_0$ and $\mathbf{f}$, such that 
\begin{align*}
    \mathcal{L}_d(\mathfrak{M}(s,t) \cap B_0) \leq  Cq^{\frac{\alpha}{n+1}
    \left(ns-\frac{3t}{2}\right)} \mathcal{L}_d(B_0), \text{ whenever }t\gg 1 \text{ and } s\leq\frac{\kappa t}{n}; \text{ and }
\end{align*}
\item for every ball $B \subseteqq U$, the following holds for all sufficiently large $t$:  $$\forall s\in \N,~\mathcal{N}_{\boldsymbol{\Theta}}(B \setminus \mathfrak{M}(s,t), s, t) \ll_{n,\mathbf{f}} q^{(d+1)t-ms} \mathcal{L}_d(B),$$ 
where $A\ll_{n,\mathbf{f}} B$ means that there exists a constant $c>0$ depending on $n,\mathbf{f}$ such that $A\leq cB.$
\end{enumerate}

\end{proposition}

\subsection{Counting Heuristics}
\label{subsec:Counting}
We now motivate the heuristic count leading to Proposition \ref{prop:Key}.
Let $Q \in \mathcal R$ with $\deg Q = t$. Then
\[
\#\{Q \in \mathcal R : \deg Q = t\}
= (q-1)q^t
\asymp q^{t+1}.
\]

For a fixed denominator $Q$ and a fixed $i\in\{1,\dots,n\}$, the number of
polynomials $P_i\in\mathcal R$ with $\deg P_i \le t$ is $q^{t+1}.$
Therefore,
\[
\#\{\mathbf P=(P_1,\dots,P_n)\in\mathcal R^n : \deg P_i \le t\}
= q^{n(t+1)}.
\]

Hence, for denominators of degree exactly $t$, the total number of rational
points of the form
$
\frac{\mathbf P}{Q}\in \mathbb F_q(T)^n$
is
\[
(q-1)q^t \cdot q^{n(t+1)}
= (q-1) q^{t + n(t+1)}
= (q-1) q^{(n+1)t + n}
\asymp q^{(n+1)t}.
\]

Outside the exceptional set $\mathfrak M(s,t)$ from
Proposition~\ref{prop:Key}, the number of rational points
$\frac{\mathbf p}{Q}$ lying within a distance $q^{-(s+t)}$
of $\mathbf f(\Delta \cap U)$ is heuristically bounded by
\[
(q^{-(s+t)})^m \, q^{(n+1)t}
= q^{-ms}q^{(d+1)t}.
\]
  
\subsection{Constraints and Differences in the Positive Characteristic Setting}\label{subsec:method}

While our approach is similar to that of \cite{BY} and \cite{beresnevich2025rational}, the function field setting poses new challenges that require a fresh perspective, given the non-monotonicity of the approximating function $\psi$. In the real case, one typically divides \([1,\infty)\) into the intervals \(\left[e^{t-1},e^t\right]\), where \(t\in \mathbb{N}\), so that every \(\psi\)-approximable vector \(\mathbf{y}\in \mathbb{R}^n\) satisfies \[\left\|\mathbf{y}-\frac{\mathbf{p}}{q}\right\|_{\infty}<\frac{\psi(e^{t-1})}{e^{t-1}}, \text{ with } \mathbf{p}\in \mathbb{Z}^n, q\in \mathbb{N}, \text{ and } e^{t-1}\leq q<e^t,\] for infinitely many \(t \in \mathbb{N}\). This reduction, which paves the way for the application of the Borel-Cantelli lemma, requires some form of monotonicity. In our scenario, the lack of any monotonicity assumption in the function field setting forces us to perform heightwise counting, i.e., we count polynomials \(Q\) with $|Q| = q^t$ (see Lemma~\ref{lem:DiagImLattPointCnt} for details). To this end, it is crucial that the main term of the upper bound for rational point counting coincides with the heuristic bound mentioned in \S \ref{subsec:Counting}. This is really difficult to avoid overcounting if the entries of diagonal flow remain in \(\mathcal{K}\), particularly with integral powers of \(T\). Moreover, another potential issue with working with integral powers is that sensible counting can only be performed for \(t\)s that belong to some proper infinite subset of \(\mathbb{N}\). This limitation results in a smaller \(\limsup\) set than what is required. \\

To circumvent this, we therefore work over the extended field  $\K\left(T^{\frac{1}{\ell}}\right)$. This, however, introduces additional challenges; for instance, the polynomial ring $\mathcal{R}$ is no longer a lattice in the extension $\K\left(T^{\frac{1}{\ell}}\right)$, but merely a discrete subgroup.  As a result, the classical tools of the geometry of numbers are not directly applicable. To overcome this obstacle, we establish results in the geometry of numbers for specific discrete subgroups that are not necessarily lattices (see \cref{sec:geometry of numbers}). Moreover, in the appendix \cref{sec:minima}, we establish a result concerning the first successive minima of such discrete subgroups over $\K\left(T^{\frac{1}{\ell}}\right)$, which we believe to be of independent interest. \\

\subsection{Structure of the Paper}
  We begin in \cref{sec:function} by recalling the definitions of analytic and $C^k$
 functions in the ultrametric setting. In \cref{sec:geometry of numbers}, we establish several auxiliary results in the geometry of numbers over function fields. Proposition~\ref{prop:Key} is proved in \S\ref{sec:analysisonparts} by deriving a measure estimate for the special part of the manifold (see \cref{subsec:Special}) and a counting estimate for the generic part (see \cref{subsec:Generic}). Finally, Proposition~\ref{prop:Key} serves as the key ingredient in the proofs of Theorems~\ref{thm:main} and~\ref{thm:DimConv}, which are presented in \cref{sec: proof of main theorem}. We conclude the paper with an additional result on the geometry of numbers over function fields, proved in \S\ref{sec:minima}.
\subsection*{Acknowledgements} 
The authors express their sincere gratitude to the University of York for hosting the workshop ``Diophantine Approximation and Related Fields,'' during which part of this work was completed. The authors thank Anish Ghosh for his continuous encouragement. The authors would like to thank Shreyasi Datta and Subhajit Jana for many helpful discussions and insightful comments. The fourth-named author is supported by the Prime Minister’s Research Fellowship (ID: 1302639) during the course of the work.

\section{Ultrametric Calculus}\label{sec:function}

In this section, we recall the concept of ultrametric $C^k$ functions from \cite[Chapter 2]{Schikhof}. Let $U$ be an open subset of $\K$, and $ g: U  \longrightarrow \K $. The \textit{first-order difference quotient} of $ g $, denoted by $ \Phi^1 g $, is defined as
\[
\Phi^1 g(x, y) \df \frac{g(x) - g(y)}{x - y}, ~\forall (x,y) \in \nabla^2 U\df \{ (x, y) \in U \times U \mid x \neq y \}.
\]
We say that $ g $ is $C^1 $ at a point $ a \in U $ if 
$
\lim_{(x,y) \to (a,a)} \Phi^1 g(x, y)
$
exists, and $g$ is called $C^1$ on $U$ if it is $C^1$ at every point of $ U $. Note that \(g\in C^1(U)\) if and only if \(\Phi^1g\) can be uniquely extended to a continuous function  \(\bar\Phi^1 g \) on \(U\times U\). The concept of a \( C^k \) function can be defined inductively. Indeed, given any \(k\in \mathbb{N}\), 
let 
\(
\nabla^k U \df \{ (x_1, \dots, x_k) \in U^k \mid x_i \neq x_j \text{ for } i \neq j \}
\). 
Define the $k$-th order difference quotient $ \Phi^k g :  \nabla^{k+1} U \longrightarrow \K$ inductively:
\[
\Phi^0 g \df g,\quad
\Phi^k g(x_1, \dots, x_{k+1}) \df \frac{\Phi^{k-1} g(x_1, x_3, \dots, x_{k+1}) - \Phi^{k-1} g(x_2, x_3, \dots, x_{k+1})}{x_1 - x_2}.
\]
One can easily see that $ \Phi^k g$ is a symmetric function of its $k+1$ variables. The function $ g $ is said to be $C^k$ at a point $a \in U$ if the following limit exists:
\[ \lim_{(x_1, \dots, x_{k+1}) \to (a, \dots, a)} \Phi^k g(x_1, \dots, x_{k+1}),\]
and $g$ is said to be $C^k$ on $U$, which is denoted by $g \in C^k(U)$, if $g$ is $C^k$ at every point of $U$. In view of \cite[Theorem 29.9]{Schikhof}, this is equivalent to $\Phi^k g$ admitting a continuous extension $\bar\Phi^k g$ to $U^{k+1}$, which must be unique  as \(\nabla^{k} U\) is dense in \(U^k\). It follows that if $g \in C^k(U)$, all derivatives up to order $k$ exist and are continuous. In fact, one has 
\[
g^{(k)}(\x) = k! \, \bar{\Phi}^k g(x, \dots, x),~\forall x\in U.
\]


Now let $g:U_1 \times \dots \times U_d \subseteqq \K^d \longrightarrow \K$, where $U_i \subseteqq \K$ are open subsets for $i=1,\dots,d$. Denote by $ \Phi_i^k g $  the $k$-th order difference quotient with respect to the $ i $-th variable. For a multi-index $ \beta = (i_1, \dots, i_d) $, we define the following: 
\[
\Phi_\beta g \df \Phi_1^{i_1} \circ \cdots \circ \Phi_d^{i_d} g.
\]
The domain of $\Phi_\beta g $ is $\nabla^{i_1+1} U_1 \times \cdots \times \nabla^{i_d+1} U_d.$ We say $g \in C^k(U_1 \times \cdots \times U_d) $ if for all multi-indices $ \beta=(i_1,\dots,i_d) $ with $ |\beta| \df \sum_{j=1}^d i_j \le k $, the difference quotient $ \Phi_\beta g $ extends continuously to a function \(
\bar{\Phi}_{\beta} g\) on \(U_1^{i_1+1} \times \cdots \times U_d^{i_d+1}\).
As in the one-variable case, for $g \in C^k(U_1 \times \dots \times U_d)$, the mixed partial derivatives
\(
\partial_\beta g \df \partial_1^{i_1} \circ \cdots \circ \partial_d^{i_d} g
\)
exist and are continuous for all $|\beta| \le k $. Indeed, it is clear that
\begin{equation} \label{beta}
\partial_\beta g(x_1, \dots, x_d) = \beta! \, \bar{\Phi}_\beta g(x_1, \dots, x_1, \dots, x_d, \dots, x_d),
\end{equation}
where each variable $x_j$ appears $ i_j + 1 $ times and $\beta! \df \prod_{j=1}^d i_j! $. 

\begin{definition} \label{definition of analytic}Let $U$ be an open subset of $\K^d$. A function $g: U\subseteqq \K^d \to \K$  is called \emph{analytic} on $U$ if for every point $ \mathbf{x}_0 \in U $, there exists an open ball $B(\mathbf{x}_0, r) \subseteqq U$
and a power series \(\sum_{\beta \in \mathbb N_{0}^d} a_{\beta} (\mathbf x - \mathbf{x}_0)^{\beta}\), with coefficients $a_{\beta} \in \K $ and $\mathbb N_0=\mathbb N \cup \{0\}$, such that
\[
g(\mathbf{x}) = \sum_{\beta \in \mathbb N_{0}^d} a_{\beta} (\mathbf x - \mathbf{x}_0)^{\beta}, \text{ for all }
 \x \in B(\x_0,r),
\]
 where, given any multi-index $\beta=(i_1,\dots,i_d)\in \mathbb N_{0}^d$ and $\mathbf y= (y_1,\dots,y_d) \in \K^d$, \(\mathbf y^{\beta} \df y_1^{i_1}\dots y_d^{i_d}\). 
In other words, $ g $ is analytic on $U$ if it is locally given by a convergent power series about every point in $U$.
\end{definition}
\noindent Note that any analytic function is of class $C^{\infty}$. Given an analytic map $\mathbf{g}:U \subseteqq \K^d\longrightarrow \K^n$ with coordinates \(g_1,\dots, g_n\), by $\nabla \mathbf{g}(\x)$ we denote the $d \times n$ matrix whose $(i,j)$-th entry is $\partial_jg_i(\x)$.  
We also recall the following second-order Taylor's formula for analytic functions (\cite[\S11-Appendix]{DBG}). 
\begin{lemma} \label{lem:Taylor} Let $U$ be an open subset of $\K^d$, and  $\x\in U$ and $\mathbf{h}\in \K^d$ are such that \(\x+\mathbf{h}\in U\). Then for any analytic $g:U\longrightarrow \K$, 
\begin{equation}\label{Taylor formula-second}
    g(\x + \mathbf{h}) = g(\x) + \sum_{i=1}^d h_i \bar\Phi_{e_i} g(\cdot)+ \sum_{\beta = (i_1, \ldots, i_d), |\beta| = 2} \bar\Phi_{\beta}g(\cdot) \prod_{j=1}^d (h_j)^{i_j},
\end{equation}
where $e_i$ is the multi-index whose $i$-th coordinate is $1$ and all other coordinates are zero, and the arguments of $\bar\Phi_{e_i} g(\cdot)$ and $\bar\Phi_{\beta}g(\cdot)$ are some of the components of $\x$ and $\mathbf{h}$. 
\end{lemma}
We also recall Besicovitch's covering theorem in the function field setup: \begin{theorem}\cite[Example 2.1]{KleinbockTomanov} \label{thm:Covering}Let $A$ be a bounded subset of $\K^n$. Then any covering by balls of $A$ has a countable subcover consisting of mutually disjoint balls.  \end{theorem}

\section{Preliminaries in Geometry of Numbers}\label{sec:geometry of numbers}
Recall that for $\ell \in \N$, $\K_{\ell} = \K(T^{1/\ell})$ and $\mathcal{R} = \F_q[T]$. Note that $\K_{\ell} = \K$ when $\ell = 1$. 
For $i = 1, \dots, n+1$ and $g \in \operatorname{SL}(n+1, \K_{\ell})$, we define the $i$-th successive minimum of the discrete subgroup $g \mathcal{R}^{n+1} \subseteqq \K_{\ell}^{n+1}$ as
\[
\lambda_i(g\mathcal{R}^{n+1}) \df \inf \left\{ \lambda > 0 : B(0, \lambda) \cap g\mathcal{R}^{n+1} 
\text{ contains } i \text{ linearly independent vectors over } \K_{\ell} \right\},
\]
where $B(0, \lambda)$ denotes the ball of radius $\lambda$ centered at the origin in $\K_{\ell}^{n+1}$. \\

We define $G \df \operatorname{SL}(n+1, \K)$ and $\Gamma \df \operatorname{SL}(n+1, \mathcal{R})$. 
Then the associated homogeneous space $\mathcal{X}_{n+1} \df G / \Gamma$ can be identified with the space of all unimodular lattices in $\K^{n+1}$. Now we recall the function field analog of Minkowski's second theorem from \cite[Equations (24) and (25)]{Mahler}.

\begin{theorem}
    \label{thm:Mink2nd}
    Let $g\in \operatorname{GL}(n+1,\K)$. Then, 
    \begin{equation}
        \prod_{i=1}^{n+1}\lambda_i(g\mathcal{R}^{n+1})=\vert \det(g)\vert.
    \end{equation}
\end{theorem}

Given a lattice $\Lambda\in \mathcal{X}_{n+1}$, we now define the \emph{dual lattice} of $\Lambda$ as follows
$$\Lambda^*\df\left\{\y\in \K^{n+1}:\forall \x\in \Lambda, \langle \x,\y\rangle\in \mathcal{R}\right\},$$
where $\langle \x,\y\rangle=\sum_{i=1}^{n+1}x_iy_i$.
We first state the duality theorem on 
$\K^{n+1}$, originally due to Mahler \cite[Equation (28)]{Mahler} (see also \cite[Lemma 5.4]{BagshawKerr}).
\begin{theorem}
\label{thm:MahDual}
Let $\Lambda$ be a lattice in $\K^{n+1}$, and $\Lambda^*$ be its dual lattice. Then we have
\begin{equation}\label{eqn:3.1}
\lambda_i(\Lambda^*)\lambda_{n+2-i}(\Lambda)=1,\ \text{for}\ 1\leq i \leq n+1. 
\end{equation}
\end{theorem}

For integers $a_1, \dots, a_{n+1}$, consider the matrix 
\begin{equation}\label{matrix-def}
    g=\operatorname{diag}\left(T^{\frac{a_1}{\ell}},\dots,T^{\frac{a_{n+1}}{\ell}}\right),
\end{equation}  and $g^*$ is defined as $g^{*}\df(g^{t})^{-1}$.  We prove the following generalization of Theorem \ref{thm:MahDual} for some specific discrete subgroups in $\K_{\ell}^{n+1}$.
\begin{theorem}[Duality]
\label{thm:Duality}
{Let $g\in \operatorname{GL}(n+1,\K_{\ell})$ be of the form \eqref{matrix-def}, and let $\Lambda$ be a lattice in $\K^{n+1}$. Then 
    
  \begin{equation}\label{duality ineq}
    q^{-\frac{\ell-1}{\ell}}  \leq \lambda_i(g^*\Lambda^*)\lambda_{n+2-i}(g\Lambda)\leq q^{\frac{\ell-1}{\ell}},~\forall i=1, \dots,n+1.
  \end{equation}}
\end{theorem}
 
\begin{proof}
Without any loss of generality, we may assume that $a_i\in \left\{0,1,\dots,\ell-1\right\}$; for if necessary, we may replace $\Lambda$ with the lattice $\operatorname{diag}\left(T^{\left\lfloor\frac{a_1}{\ell}\right\rfloor},\dots, T^{\left\lfloor{\frac{a_{n+1}}{\ell}}\right\rfloor}\right)\Lambda$, where $\lfloor c\rfloor\df\max \{k \in \mathbb Z: k \leq c\}$ for $c\in \mathbb R$. 
  
Since $a_i\in\{0,1,\dots,\ell-1\}$, it follows that

   \begin{align}
       \|\mathbf{v}\|=\|g^{-1}g\mathbf{v}\|\leq \|g^{-1}\|\|g\mathbf{v}\|\leq \|g\mathbf{v}\|\leq \|g\|\|\mathbf{v}\|\leq q^{\frac{\ell-1}{\ell}}\|\mathbf{v}\|;\text{ and }\nonumber \\
        q^{-\frac{\ell-1}{\ell}}\|\mathbf{v}\|\leq \|g\|^{-1}\|\mathbf{v}\|= \|g\|^{-1}\|gg^{-1}\mathbf{v}\|\leq \|g^{-1}\mathbf{v}\|\leq \|\mathbf{v}\|, \forall \mathbf{v}\in\K^{n+1},\label{second ineq}
   \end{align} where, for any diagonal matrix \(D\df  \operatorname{diag}(\beta_1, \dots ,\beta_{n+1})\in \operatorname{Mat}_{n+1}(\K_{\ell}) \), \(\|D\|\) is defined by \(\max_{1\leq i\leq n+1}|a_i|\). Since $\K_{\ell}$ is a finite extension over $\K$, $\mathbf{v}_1, \dots, \mathbf{v}_{n+1}\in \Lambda$  are linearly independent over $\K$ if and only if $\mathbf{v}_1, \dots, \mathbf{v}_{n+1}$ are linearly independent over $\K_{\ell}$. Hence, 
\begin{equation}\label{linearly indep}
    \begin{split}
    \mathbf{v}_1, \dots, \mathbf{v}_{n+1} \text{ are linearly independent}
    \Longleftrightarrow g\mathbf{v}_1, \dots, g\mathbf{v}_{n+1} \text{ are linearly independent.}\\ \Longleftrightarrow g^{-1}\mathbf{v}_1, \dots, g^{-1}\mathbf{v}_{n+1} \text{ are linearly independent.}
\end{split}
\end{equation}
Consider linearly independent vectors $\mathbf{v}_1, \dots, \mathbf{v}_{n+1}$ in $\Lambda$ such that $\|\mathbf{v}_i\|=\lambda_i(\Lambda)=\lambda_{n+2-i}^{-1}\left(\Lambda^{\ast}\right)$, for all $i=1, \dots, n+1$.  Then from \eqref{linearly indep} and \eqref{second ineq}, it follows that 
   \begin{equation}\label{succ:1}
       \lambda_i(g\Lambda)\leq  q^{\frac{\ell-1}{\ell}}\lambda_i(\Lambda),~\forall i=1, \dots, n+1. 
   \end{equation} Similarly, one obtains that 
   \begin{align}
      \lambda_i(\Lambda)\leq   \lambda_i(g\Lambda), \label{succ:2}\\
  q^{-\frac{\ell-1}{\ell}}\lambda_i(\Lambda)\leq \lambda_i(g^{-1}\Lambda), \label{succ:3} \text{ and }\\
  \lambda_i(g^{-1}\Lambda)\leq \lambda_i(\Lambda),~\forall i=1, \dots,n+1.\label{succ:4}
   \end{align}
   Since $\Lambda$ is arbitrary, by replacing $\Lambda$ with $\Lambda^*$, from \eqref{succ:3} and \eqref{succ:4}, we also have 
   \begin{equation}\label{comp:1}
     q^{-\frac{\ell-1}{\ell}}\lambda_i(\Lambda^{\ast})\leq \lambda_i(g^{\ast}\Lambda^{\ast})\leq \lambda_i(\Lambda^{\ast}),~\forall i=1, \dots,n+1,
   \end{equation} since $g$ is diagonal, so that $g^{-1}$ is the same as its transpose. Multiplying inequalities in 
   \eqref{succ:1} with $i$ replaced by $n+2-i$ and \eqref{comp:1} now yields that 
   \[\lambda_i(g^{\ast}\Lambda^{\ast})\lambda_{n+2-i}(g\Lambda)\leq q^{\frac{\ell-1}{\ell}}\lambda_i(\Lambda^{\ast})\lambda_{n+2-i}(\Lambda),~\forall i=1, \dots, n+1. 
   \]
   On the other hand, using \eqref{succ:2} and \eqref{comp:1}, we obtain,
\[q^{-\frac{\ell-1}{\ell}}\lambda_i(\Lambda^{\ast})\lambda_{n+2-i}(\Lambda)\leq \lambda_i(g^{\ast}\Lambda^{\ast})\lambda_{n+2-i}(g\Lambda).\]
Combining these, we have,
\[q^{-\frac{\ell-1}{\ell}}\lambda_i(\Lambda^{\ast})\lambda_{n+2-i}(\Lambda)\leq \lambda_i(g^{\ast}\Lambda^{\ast})\lambda_{n+2-i}(g\Lambda)\leq q^{\frac{\ell-1}{\ell}}\lambda_i(\Lambda^{\ast})\lambda_{n+2-i}(\Lambda),~\forall i=1, \dots, n+1.\]
   Equation \eqref{duality ineq} is immediate from this, due to \eqref{eqn:3.1}.
   \end{proof} 
   \begin{remark}
   {When $\ell=1$, \eqref{duality ineq} reduces to \eqref{eqn:3.1}.}
\end{remark}

Recall from \cite{Mahler} that convex bodies in $\K^{n+1}$ are precisely the sets of the form $g\mathcal{O}_{\K}^{n+1}$, where $g\in \operatorname{GL}(n+1,\K_{\ell})$. Then, clearly $\operatorname{Vol}(g \mathcal{O}_{\K}^{n+1})=\vert \det(g)\vert$. We now state the following lemma, due to Bagshaw and Kerr \cite{BagshawKerr}, which counts the number of lattice points in a convex body:

\begin{lemma}{\cite[Lemma 6.2]{BagshawKerr}}
\label{lem:BKCount}
   Let $u\in \operatorname{GL}(n+1,\K)$, let $\Lambda=u{\mathcal{R}}^{n+1}$, and let $\mathcal{C}=h\mathcal{O}_{\K}^{n+1}$ be a convex body. Then, 
   $$\#(\Lambda\cap \mathcal{C})=\prod_{i=1}^{n+1}\left\lceil\frac{q}{\lambda_i(h^{-1}\Lambda)}\right\rceil,$$
   where $\lceil\cdot\rceil:\mathbb{R}\rightarrow \mathbb{Z}$ is defined by $\lceil \alpha\rceil=\min\{k\in \Z:k\geq \alpha\}$. In particular, if $\mathcal{C}$ contains a fundamental domain for $\Lambda$, then, 
   $$\#(\Lambda\cap \mathcal{C})=q^{n+1}\frac{\operatorname{Vol}(\mathcal{C})}{\vert \det(u)\vert}.$$
\end{lemma}
We generalize the above result for specific discrete subgroups of the form $g\Lambda \subseteqq \K_{\ell}^{n+1}$ with $g$ being diagonal over $\K_{\ell}$ and $\Lambda$ being a lattice in $\K$  under the assumption $ \lambda_{n+1}(g\Lambda) \leq q^{c} $ for some $c > 0$. To proceed, we first need to establish the following straightforward lemma: 
\begin{lemma}\label{fundamental domain}
    Let $n \in \N$ and $\Lambda \subseteqq \K^{n+1}$ be a lattice. Suppose that $\mathbf{v}_1, \dots, \mathbf{v}_{n+1}$ are linearly independent vectors in $\Lambda$. Then there exists a basis $\mathbf{w}_1, \dots, \mathbf{w}_{n+1}$ of $\Lambda$ such that
\[
        \mathbf{v}_1=a_{11}\mathbf{w}_1,
        \mathbf{v}_2=a_{21}\mathbf{w}_1+ a_{22}\mathbf{w}_2, \dots, 
        \mathbf{v}_{n+1}=a_{(n+1)1}\mathbf{w}_1+\dots +a_{(n+1)(n+1)}\mathbf{w}_{n+1},\] where $a_{ij}$s are from $\mathcal{R}$, satisfying the following:
    \begin{equation}\label{coeff relation}
        a_{ii}\neq 0, |a_{ii}|\geq |a_{ij}|,\text{ for }i=1, \dots, n+1, j=1, \dots, i.
    \end{equation}
\end{lemma}
Lemma \ref{fundamental domain} can easily be proved by adapting the argument given in the proofs of \cite[ Theorem 1 and Corollary 1, \S 1.2]{cassels1971}. As the argument therein is quite elementary, we leave the details to the reader. However, the conditions mentioned above in \eqref{coeff relation} are important for us. A simple induction argument shows that \[\{\mathbf{w}_1, \dots, \mathbf{w}_{n+1}\}\subseteqq\mathcal{O}_{\K}\mathbf{v}_1+ \mathcal{O}_{\K}\mathbf{v}_2+ \dots +\mathcal{O}_{\K}\mathbf{v}_{n+1}.\]Consequently, the parallelepiped spanned by the vectors $\mathbf{v}_1, \dots, \mathbf{v}_{n+1}$ contains a fundamental domain of the lattice $\Lambda$. 
\begin{lemma}
\label{lem:DiagImLattPointCnt}
 
    Let $g=\operatorname{diag}\left(T^{a_1+\frac{\alpha_1}{\ell}},\dots,T^{a_{n+1}+\frac{\alpha_{n+1}}{\ell}}\right)\in \operatorname{GL}(n+1,\K_{\ell})$, where $a_i\in \mathbb{Z}$ and $\alpha_i=0,1,\dots,\ell-1$. Consider any $c \in \mathbb{R}$ such that $\lambda_{n+1}(g\Lambda)\leq q^{c}$, where $\Lambda=u\mathcal{R}^{n+1}$ and $u\in \operatorname{GL}(n+1,\K)$. Then,
    $$\#\left\{\mathbf{v}\in g\Lambda:\Vert \mathbf{v}\Vert\leq q^{c}\right\}\leq \frac{q^{(n+1)(c+1)}}{\vert \det(g)\vert\cdot \vert \det(u)\vert}.$$
    \end{lemma}
\begin{proof}
    Denote $[g]=\operatorname{diag}(T^{a_1},\dots,T^{a_{n+1}})$ and $\{g\}=\operatorname{diag}\left(T^{\frac{\alpha_1}{\ell}},\dots,T^{\frac{\alpha_{n+1}}{\ell}}\right)$. Since $g=\{g\}[g]$ and $  \vert \det(\{g\})\vert\geq 1$, clearly $\vert \det(g)\vert\geq \vert \det([g])\vert$.
   \begin{equation}\label{obs}
       \Vert [g]\mathbf{v}\Vert=\Vert \{g\}^{-1}{g}\mathbf{v}\Vert\leq\Vert{g}\mathbf{v}\Vert,\text{ for all }\mathbf{v}\in \Lambda.
   \end{equation}
   Let $g\mathbf{v}_1,\dots,g\mathbf{v}_{n+1}$ be linearly independent vectors in $g\Lambda$ with $\Vert g\mathbf{v}_i\Vert=\lambda_i(g\Lambda)$. Clearly, $[g]\mathbf{v}_1,\dots,[g]\mathbf{v}_{n+1}$ are linearly independent. Then it follows from \eqref{obs} and the hypothesis $\lambda_{n+1}(g\Lambda)\leq q^{c}$ that $\lambda_{n+1}([g]\Lambda)\leq  q^{c}$. Furthermore, \eqref{obs} yields that the parallelepiped spanned by the vectors $[g]\mathbf{v}_1,\dots,[g]\mathbf{v}_{n+1}$ inside $\K^{n+1}$ is actually contained in $T^{\lfloor{c}\rfloor}\mathcal{O}_{\K}^{n+1}$, where $\lfloor c\rfloor=\max\{n\in \mathbb{Z}:n\leq c\}$. As a consequence,  $T^{\lfloor{c}\rfloor}\mathcal{O}_{\K}^{n+1}$ contains a fundamental domain of $\Lambda$. Thus, by \eqref{obs} and Lemma \ref{lem:BKCount}, 

    \[
        \#\left\{\mathbf{v}\in g\Lambda:\Vert \mathbf{v}\Vert\leq q^{c}\right\}\leq \#\{\mathbf{v}\in [g]\Lambda:\Vert \mathbf{v}\Vert\leq q^{\lfloor{c}\rfloor}\}
        \leq \frac{q^{(n+1)(c+1)}}{\vert \det([g])\vert\cdot \vert\det(u)\vert} \leq\frac{q^{(n+1)(c+1)}}{\vert \det(g)\vert\cdot \vert \det(u)\vert}.
    \]
     
\end{proof}
\section{Analysis on generic and special part}\label{sec:analysisonparts}
\subsection{Some Preliminary Estimates}
\label{sec:Ests}
In this subsection, we provide several estimates that will be used to prove Theorems \ref{thm:main} and \ref{thm:DimConv}. Recall that the manifold $\mathscr{M}$ is defined by the map $\mathbf{f}: U\subseteqq \K^d \to \K^n$, where
$$
\mathbf{f} \df (x_1, \ldots, x_d, f_1(\x), \ldots, f_m(\x)) = (\x, \bm{f}(\x)), \,\, \text{for}\,\, \x=(x_1,\dots,x_d) \in U
$$
with $d = \dim \mathscr{M}$, $m= \codim \mathscr{M} $, and $U\subseteqq \K^d$ being an open set. We also assume that the defining map $\mathbf{f}$ is analytic. 
\begin{lemma}
\label{lem:Qf(x)-(P+thetaSmall)}
    Let $s,t \in \mathbb N$ and suppose that for some $\x\in U$, we have $$\left\lVert \mathbf{f}(\x)-\frac{\mathbf{P}+\boldsymbol{\Theta}}{Q}\right\rVert<q^{-(s+t)},$$ where  $\mathbf{P}=(P_1,\cdots,P_n)\in \mathcal{R}^n,$ $Q \in \mathcal{R} \setminus \{0\}$ with $\vert Q\vert=q^t$ and $\boldsymbol{\Theta}=(\Theta_1,\cdots,\Theta_n)$. Then, 
    \begin{enumerate}[label=($\mathfrak{E}$.\arabic*)]
        \item \label{eqn:i=1...dQf-(P+Theta)} $\vert Qx_i-P_i-\Theta_i\vert<q^{-s}$ for every $i=1,\dots,d$.
        \item \label{eqn:j=1...mQf-P+Theta} $\vert Qf_j(\x)-\sum_{i=1}^d\partial_if_j(\x)(Qx_i-P_i-\Theta_i)-P_{d+j}-\Theta_{d+j}\vert<Mq^{-s}$ for every $j=1,\dots,m$.
    \end{enumerate}
\end{lemma}
\begin{proof}

\ref{eqn:i=1...dQf-(P+Theta)} is easily seen by considering the $i$-th coordinate of 
$\mathbf{f}(\x) - \frac{\mathbf{P} + \boldsymbol{\Theta}}{Q}$ for $i = 1, \dots, d$, 
and using $|Q| = q^t$. To prove \ref{eqn:j=1...mQf-P+Theta},  using the ultrametric inequality, for every $j=1,\dots,m$, we have
\[\begin{split}
\bigg| Q f_j(\x)
 - \sum_{i=1}^d \partial_i f_j(\x) \big(Qx_i - P_i - \Theta_i\big)
 - P_{d+j} - \Theta_{d+j} \bigg|
& \leq 
 \max \Big\{
 |Q f_j(\x) - P_{d+j} - \Theta_{d+j}|,\nonumber\\  
 &\max_{i=1,\dots,d} |\partial_i f_j(\x)| \cdot |Qx_i - P_i - \Theta_i|
 \Big\}. 
\end{split}\]
\noindent Since $
\big\|\mathbf{f}(\x) - \tfrac{\mathbf{P} + \boldsymbol{\Theta}}{Q}\big\| < q^{-(s+t)}
\,\, \text{and} \,\, |Q| = q^t
$, 
it follows that $|Q f_j(\x) - P_{d+j} - \Theta_{d+j}| < q^{-s}$. 
Combining this and 
\ref{eqn:i=1...dQf-(P+Theta)}, we get
\[
\max \left\{
|Q f_j(\x) - P_{d+j} - \Theta_{d+j}|,
\max_{i=1,\dots,d} |\partial_i f_j(\x)| \cdot |Qx_i - P_i - \Theta_i|
\right\}
< M q^{-s}.
\]
This proves  \ref{eqn:j=1...mQf-P+Theta}. 
    
\end{proof}
\begin{lemma}
\label{lem:NearPnts}
    Suppose that $\x\in U$ and $(\mathbf{P},Q)\in \mathcal{R}^{n} \times (\mathcal{R}\setminus\{0\})$ satisfy the hypothesis of Lemma \ref{lem:Qf(x)-(P+thetaSmall)}. Let $\x_0\in B\left(\x,q^{-\frac{s+t}{2}}\right)$, where $\x_0=(x_{1,0},\cdots,x_{d,0})\in U$. If $t>s>0$, then
    \begin{enumerate}[label=(\roman*)]
        \item $\vert Qx_{i,0}-P_i-\Theta_i\vert<q^{\frac{t-s}{2}}$ for $i=1,\dots,d$; and 
        \item For every $j=1,\dots,m$, we have $$\left|Qf_j(\x_0)-\sum_{i=1}^d\partial_if_j(\x_0)(Qx_{i,0}-P_i-\Theta_i)-P_{d+j}-\Theta_{d+j}\right|\leq Mq^{-s}.$$
    \end{enumerate}
\end{lemma}
\begin{proof}
    Let $\x_0=(x_{1,0},\cdots,x_{d,0})\in U$. By Lemma \ref{lem:Qf(x)-(P+thetaSmall)}, we see that the vector $\x$ satisfies \ref{eqn:i=1...dQf-(P+Theta)} and \ref{eqn:j=1...mQf-P+Theta}. Hence, 
    \[
        \vert Qx_{i,0}-P_i-\Theta_i\vert \leq \max\left\{\vert Qx_i-P_i-\Theta_i\vert,\vert Q\vert\cdot \vert x_i-x_{i,0}\vert\right\} \leq \max\left\{q^{-s},q^{\frac{t-s}{2}}\right\}=q^{\frac{t-s}{2}}.\]
  Also,  for every $j=1,\dots, m$,  it is easy to see that

 \begin{equation*}
    \begin{split}
        \left|Qf_j(\x_0)-\sum_{i=1}^d\partial_if_j(\x_0)(Qx_{i,0}-P_i-\Theta_i)-P_{d+j}-\Theta_{d+j}\right|\\\leq 
        \max\Bigg\{\left|Qf_j(\x)-\sum_{i=1}^d\partial_if_j(\x)(Qx_i-P_i-\Theta_i)-P_{d+j}-\Theta_{d+j}\right|,\\
        \vert Q\vert \cdot \left|f_j(\x_0)-f_j(\x)-\sum_{i=1}^d\partial_if_j(\x)(x_{i,0}-x_i)\right|,\\
       \max_{i=1,\dots,d}\vert \partial_if_j(\x_0)-\partial_if_j(\x)\vert\cdot \vert Qx_i-P_i-\Theta_i\vert
        \Bigg\}, \\
        \leq \max\left\{M q^{-s},q^t M \left(q^{\frac{-t-s}{2}}\right)^2,Mq^{-s}\right\} =Mq^{-s},
        \end{split} 
    \end{equation*} in view of \ref{eqn:j=1...mQf-P+Theta},  \ref{IV}, and Taylor's formula  \eqref{Taylor formula-second}.
\end{proof}
\subsection{Analysis on the Generic Part}
\label{subsec:Generic} Recall that $\K_{2(n+1)}$ denotes the extended field $\K\left(T^{\frac{1}{2(n+1)}}\right).$
We first define the generic part. To do so, we first define the diagonal flow $g_{s,t} \in \operatorname{SL}\left(n+1,\mathcal{K}_{2(n+1)}\right)$ as
$$g_{s,t} \df\operatorname{diag}\left\{\underbrace{T^{\frac{(d+2)t+2s}{2(n+1)}},\dots, T^{\frac{(d+2)t+2s}{2(n+1)}}}_m,\overbrace{T^{\frac{(d+2)t+2s}{2(n+1)}-\frac{s+t}{2}},\dots,T^{\frac{(d+2)t+2s}{2(n+1)}-\frac{s+t}{2}}}^d, T^{\frac{(d+2)t+2s}{2(n+1)}-s-t}\right\}.$$ For $\x\in U$, denote by $J(\x)$ the Jacobian matrix of $\mathbf{f}$ at $\x$, i.e., $$J(\x)\df \left( \frac{\partial f_i}{\partial x_j}(\x) \right)_{1\leq i\leq m, 1\leq j\leq d}.$$
Given $k\in \N$, we let $\sigma_k$ stand for the following $k \times k$ matrix: 
\[\sigma_{k}\df\begin{pmatrix}
    0 & 0 & \hdots & 0 & 1\\
    0 & 0 & \hdots & 1 & 0\\
    \vdots&\vdots&\ddots& \vdots & \vdots\\
    0 & 1 & \hdots & 0 & 0\\
    1 & 0 & \hdots & 0 & 0
    \end{pmatrix}.\]
    Observe that $\sigma_k$ acts on row vectors from the right and on column vectors from the left by reversing the order of their coordinates. $\sigma_k$ is an involution. Clearly $\sigma_k^{-1}=\sigma_k$. We now also consider the following matrices in $\operatorname{SL}(n+1,\K)$:
$$z(\x) \df\begin{pmatrix}
    \mathrm{I}_m&-\sigma_m^{-1}J(\x)\sigma_d&0\\
    0&\mathrm{I}_d&0\\
    0&0&1
\end{pmatrix},~u(\x) \df \begin{pmatrix}
    \mathrm{I}_n&\sigma_n^{-1}\mathbf{f}(\x)^t\\
    0 &  1\\
\end{pmatrix},\text{ and }u_1(\x) \df z(\x)u(\x).$$

It can be seen that $u_1(\x)$ has the following form:
\begin{equation*}
    u_1(\x)=\begin{pmatrix}
        1&0&\dots&0&-\partial_df_m(\x)&\dots&-\partial_1f_m(\x)&f_m(\x)-\sum_{i=1}^dx_i\partial_if_m(\x)\\
        \vdots&\vdots&\vdots&\vdots&\vdots&\vdots&\vdots&\vdots\\
        0&0&\vdots&1&-\partial_df_1(\x)&\dots&-\partial_1f_1(\x)&f_1(\x)-\sum_{i=1}^dx_i\partial_if_1(\x)\\
        0&0&\dots&0&1&0&\dots&x_d\\
        \vdots&\vdots&\vdots&\vdots&\vdots&\vdots&\vdots&\vdots\\
        0&0&\dots&0&0&\dots&1&x_1\\
        0&0&\dots&0&0&\dots&0&1
    \end{pmatrix}
\end{equation*}
We define a new operator $^{\star}$ on $\operatorname{GL}(n+1,\K_{\ell})$ by  
\[
g^{\star} \df \sigma_{n+1}^{-1} \left( g^t \right)^{-1} \sigma_{n+1} \quad \text{for } g \in\operatorname{GL}(n+1,\K_{\ell}).
\] Note that given any $g_1,g_2 \in \operatorname{GL}(n+1,\K_{\ell}),$ one has $(g_1 g_2)^{\star}=g_1^{\star}g_2^{\star}$. Since $\sigma_{n+1}$ acts by a mere permutation of the coordinates, for any $g\in \operatorname{GL}(n+1,\K_{\ell})$, we have $\lambda_i\left(\left(\sigma_{n+1}^{-1}g\sigma_{n+1}\right)\mathcal{R}^{n+1}\right)=\lambda_i(g\mathcal{R}^{n+1})$ for every $i=1,\dots,n+1$. By replacing $g$ with $\left(g^t\right)^{-1}$, for every $i=1,\dots,n+1$, we have $\lambda_i\left(g^{\star}\mathcal{R}^{n+1}\right)=\lambda_i\left(g^*\mathcal{R}^{n+1}\right)$, where $g^*=(g^t)^{-1}$ is the dual of $g$ as defined in \cref{sec:geometry of numbers}.\\

Define \begin{equation}\label{Mnew}
\begin{split}
    \mathfrak{M}_0(s,t) \df\left\{\x\in U:\lambda_{1}(g_{s,t}^{\star}u_1^{\star}(\x)\mathcal{R}^{n+1})<q^{-\frac{(d+2)t-2ns}{2(n+1)}}\right\}\\= \left\{\x\in U:\lambda_{1}(g_{s,t}^*u_1^*(\x)\mathcal{R}^{n+1})<q^{-\frac{(d+2)t-2ns}{2(n+1)}}\right\},
\end{split}
\end{equation}
and
$$\mathfrak{M}(s,t)=\bigcup_{\x\in \mathfrak{M}_0(s,t)}B\left(\x,q^{-\frac{s+t}{2}}\right).$$ 
Henceforth, we define the \emph{generic part} of the manifold as $U\setminus \mathfrak{M}(s,t)$ and the \emph{special part} as $\mathfrak{M}(s,t)$. \\

We first estimate the number of points in the manifold's generic part. Towards this end, we prove the following proposition.
\begin{proposition}\label{prop:counting generic}
Let $\mathbf{f}$ and $U$ be as in Proposition \ref{prop:Key}. Then, given any ball $B \subseteqq U$, the following holds for sufficiently large $t\in \N$: 
\begin{equation}\label{eqn:GenCnt}\mathcal{N}_{\boldsymbol{\Theta}}(B \setminus \mathfrak{M}(s, t);s,t)\ll q^{(d+1)t}q^{-ms}\mathcal{L}_d(B),~\forall s\in \N.
\end{equation}
\end{proposition}
The proof of Proposition \ref{prop:counting generic} relies on the following counting estimate:

\begin{lemma}
\label{lem:GenericCNT} Let $\mathbf{f}$ and $U$ be as in Proposition \ref{prop:Key}. Consider $s, t\in\N$, and an open ball  $B\subseteqq U\subseteqq \K^d$. Then, for all $\x_0\in (U\setminus\mathfrak{M}(s,t))\cap B$,
    $$\mathcal{N}_{\boldsymbol{\Theta}}\left(B\left(\x_0,q^{-\frac{s+t}{2}}\right)\cap B; s,t\right)\ll q^{\frac{d}{2}(t+s)+(t-ns)}.$$
\end{lemma}
\begin{proof}
\color{black}
   Let $\mathbf{x}_0=(x_{1,0},\dots,x_{d,0})$. Without loss of generality, assume that $\mathcal{N}_{\boldsymbol{\Theta}}\left(B(\x_0,q^{-\frac{s+t}{2}})\cap B;s,t\right)$ is greater than \(1\). Then for $i=1,2$ there exist $(\mathbf{P}_i,Q_i)\in \mathcal{R}^{n} \times \mathcal{R} \setminus \{0\}$ and $\x_i\in B\left(\x_0,q^{-\frac{s+t}{2}}\right)$, such that 
    $\left\Vert\mathbf{f}(\x_i)-\frac{\mathbf{P}_i+\boldsymbol{\Theta}}{Q_i}\right\Vert<q^{-(s+t)}$ with $|Q_i|=q^t.$
    Hence, by Lemma \ref{lem:NearPnts}, we have for $h=1,2$, and for every $k=1,\dots, d,$, we have
    \begin{equation}
    \label{eqn:i=1...d}
        \vert Q_{h}x_{k,0}-P_{h,k}-\Theta_k\vert<\max\left\{q^{-s},q^{t}q^{-\frac{s+t}{2}} \right\}=q^{\frac{t-s}{2}}.
    \end{equation}
    Now, for every $j=1,\dots,m$, we have 
    \begin{equation}
    \begin{split}
    \label{eqn:j=1...m}
        \left|Q_{h}f_j(\x_0)-\sum_{i=1}^d\partial_if_j(\x_0)(Q_{h}x_{i,0}-P_{h,i}-\Theta_i)-P_{h,d+j}-\Theta_{d+j}\right|
        \leq Mq^{-s}.
    \end{split}
    \end{equation}
    Let $Q=Q_2-Q_1$ and $\mathbf{P}=\mathbf{P}_2-\mathbf{P}_1=(P_1,\dots,P_n)$ so that $\vert Q\vert\leq q^t$. Hence, by subtracting \eqref{eqn:i=1...d} from each other with $h=1,2$, we have
    \begin{equation}
        \vert Qx_{i,0}-P_i\vert< q^{\frac{t-s}{2}} \,\,\,\, \text{for } i=1,\dots,d
    \end{equation} Similarly, by subtracting \eqref{eqn:j=1...m} from each other with $h=1,2$, we have
    \begin{equation}
    \begin{split}
        \left|Qf_j(\x_0)-\sum_{i=1}^d\partial_if_j(\x_0)(Qx_{i,0}-P_i)-P_{d+j}\right|\leq  M q^{-s}\quad \text{for } j=1,\dots,m.
    \end{split}
    \end{equation}
    Given $r>0$ and $k\in \N$, define $[r]^k=\{\mathbf{v}\in \K_{2(n+1)}^k:\Vert \mathbf{v}\Vert\leq r\}$.  
    Thus, we have
$$u_1(\x_0)\begin{pmatrix}
        -\mathbf{P}\sigma_n\\
        Q
    \end{pmatrix}\in \left(\left[Mq^{-s}\right]^m\times \left[q^{\frac{t-s}{2}}\right]^d\times [q^t]\right)\cap u_1(\x_0)\mathcal{R}^{n+1}.$$

Thus, 
    \begin{equation*}
    \label{eqn:g_s,tu_q(x_0)(-P,Q)}
    g_{s,t}u_1(\x_0)\begin{pmatrix}
        -\mathbf{P}\sigma_n\\
        Q
    \end{pmatrix}\in {\left( M\left[q^{\frac{(d+2)t-2ns}{2(n+1)}}\right]^{n+1}\right)} \cap g_{s,t}u_1(\x_0)\mathcal{R}^{n+1}.\end{equation*}
   
   Hence, one has $\lambda_1(g_{s,t}^{\star}u_1^{\star}(\x_0)\mathcal{R}^{n+1})=\lambda_1(g_{s,t}^*u_1^*(\x_0)\mathcal{R}^{n+1})\geq q^{-\frac{(d+2)t-2ns}{2(n+1)}}.$ Now, it follows from duality (Theorem \ref{thm:Duality}) that $\lambda_{n+1}(g_{s,t}u_1(\x_0)\mathcal{R}^{n+1})\ll_n q^{\frac{(d+2)t-2ns}{2(n+1)}}$.  Therefore, by applying Lemma \ref{lem:DiagImLattPointCnt} with $c=\frac{(d+2)t-2ns}{2(n+1)}
    $,  we have
  \[
\#g_{s,t}u_1(\x_0)\mathcal{R}^{n+1}\cap M\left[q^{\frac{(d+2)t-2ns}{2(n+1)}}\right]^{n+1}\ll \frac{q^{\frac{(d+2)t-2ns}{2}}}{q^{\frac{(d+2)t+2s}{2}-\frac{d+2}{2}(s+t)}}
=q^{\frac{d}{2}(t+s)+(t-ns)}.
\]\end{proof}
The following lemma, coupled with Lemma \ref{lem:GenericCNT}, now yields Proposition \ref{prop:counting generic}:
\begin{lemma}\label{lem:lemma counting1}
For all sufficiently large $t \in \N$, we have the following for any $s\in \N$: 
\[\mathcal{N}_{\boldsymbol{\Theta}}(B \setminus \mathfrak{M}(s, t); s, t) \leq q^{\frac{d(t+s)}{2}}\mathcal{L}_d(B) \max_{\x_0\in B\setminus \mathfrak{M}(s,t)}
\mathcal{N}_{\boldsymbol{\Theta}}\left(B\left(\x_0,q^{-\frac{s+t}{2}}\right)\cap B; s,t\right).\]
\end{lemma}
\begin{proof}

Choose $t\in \N$ large enough so that $q^{-\frac{t}{2}}$ is at most the radius of $B$. Then any open ball centred at a point in $B$ and having radius not exceeding $q^{-\frac{t}{2}}$ must be contained in $B $. From this, it follows that, when $t\gg1$, for all $s\in \N$, $B$ can be written as a disjoint union of open balls having radius $q^{-\frac{s+t}{2}}$, and clearly the number of such balls in the union is at most  $q^{\frac{d}{2}(s+t)}\mathcal{L}_d(B)$. If any of these balls intersects $\mathfrak{M}(s,t)$, then that ball coincides with $B\left(\x_0,q^{-\frac{s+t}{2}}\right)$ for some $\x_0\in \mathfrak{M}(s,t)$.   Therefore, 
$$\mathcal{N}_{\boldsymbol{\Theta}}(B\setminus \mathfrak{M}(s,t); s,t)\leq q^{\frac{d}{2}(t+s)}\mathcal{L}_d(B)\max_{\x_0\in B\setminus \mathfrak{M}(s,t)}\mathcal{N}_{\boldsymbol{\Theta}}\left(B\left(\x_0,q^{-\frac{s+t}{2}}\right)\cap B; s,t\right).$$
\end{proof}

\subsection{Analysis on the Special Part}
\label{subsec:Special}
To deal with the \emph{special part} of the manifold, for $r,s',t_1,\dots,t_n\in \mathbb{Z}$ and for any ball $B\subseteqq \K^d$, define
$$ \mathfrak{G}_{\mathbf{f}}(r,s',t_1,\dots,t_n)=\left\{\x\in B:\exists\,\, \mathbf{P}=(P_1,\dots,P_n)\in \mathcal{R}^n,Q\in \mathcal{R}\left|\begin{array}{lcl}
 \vert \mathbf{f}(\x)\cdot \mathbf{P}+Q\vert<q^{-r}\\[2ex]
 \Vert \nabla(\mathbf{f}(\x)\cdot \mathbf{P})\Vert<q^{s'} \\[2ex] 
 \vert P_i\vert<q^{t_i},\,\,\,i=1,\dots,n
\end{array} \right.\right\}.$$
We use the following theorem by Das and Ganguly \cite{das2022inhomogeneous}, which bounds the measure of the set $\mathfrak{G}_{\mathbf{f}}(r,s',t_1,\dots,t_n)$.
\begin{theorem}[{\cite[Theorem 6.1]{das2022inhomogeneous}}]\label{BKM Function Field}
Suppose $U$ is an open subset of $\K^d$ and $\mathbf{f}: U \rightarrow \K^n$ which satisfies \ref{I}-\ref{IV}. Then for any $\x_0 \in U,$ one can find a neighbourhood $V \subseteqq U$ of $\x_0$ and $\alpha >0$ with the following property: for any ball $B \subseteqq V,$ there exists $E>0$ such that for any choice of $r,s',t_1,\dots,t_n \in \mathbb Z$ with $r \geq 0,$  $t_1,\dots,t_n \geq 1,$ and  $s' + \sum_i t_i - r - \max_i t_i < 0$ one has 

\begin{equation*}\label{equ:small grad}
 \mathcal{L}_d\left(\mathfrak{G}_{\mathbf{f}}(r,s',t_1,\dots,t_n)\right) \leq E \gamma^{\alpha}\mathcal{L}_d(B), \text{ where  }\gamma \df \max \left( q^{-r}, q^{\frac{s' + \sum_i t_i - r - \max_i t_i}{n+1}} \right).
\end{equation*} 
\end{theorem}

\begin{lemma}
\label{lem:specCont}
Let $\mathbf{f}$ and $U$ be as in Proposition \ref{prop:Key}. Then  there exists a constant $\widetilde{C}\in \mathbb N$, possibly depending on $M, d$ and $n$,  such that for all $t\in \N$,  we have
\begin{equation*}
    \begin{split}
        \mathfrak{M}(s,t)\subseteqq\mathfrak{G}_{\mathbf{f}}\left(\widetilde{C}t,\widetilde{C}\lceil(s-t)/2\rceil,\widetilde{C}s,\dots,\widetilde{C}s\right).
    \end{split}
    \end{equation*}
\end{lemma}
\begin{proof}
Let $\x_0\in\mathfrak{M}_0(s,t)$. Then,
\begin{equation}\label{special1}\lambda_1(g_{s,t}^{\star}u_1^{\star}(\x_0)\mathcal{R}^{n+1})<q^{-\frac{(d+2)t-2ns}{2(n+1)}}.
\end{equation}
Note that
    $$g_{s,t}^{\star}=\operatorname{diag}\left(T^{s+t-\frac{(d+2)t+2s}{2(n+1)}},\underbrace{T^{\frac{s+t}{2}-\frac{(d+2)t+2s}{2(n+1)}},\dots,T^{\frac{s+t}{2}-\frac{(d+2)t+2s}{2(n+1)}}}_d,\overbrace{T^{-\frac{(d+2)t+2s}{2(n+1)}},\dots,T^{-\frac{(d+2)t+2s}{2(n+1)}}}^m\right)$$
   and $u_1^{\star}(\x_0)=\begin{pmatrix}
        1&-\x_0&-\bm{f}(\x_0)\\
        &\mathrm{I}_d&J(\x_0)\\
        &&\mathrm{I}_m
    \end{pmatrix}$. Then, by \eqref{special1}, there exists some $(Q_0,\mathbf{P})\df(Q_0,\mathbf{P}^{(d)},\mathbf{P}^{(m)})=\left(Q_0,P_{1}^{(d)},\dots,P_{d}^{(d)},P_{1}^{(m)},\dots,P_{m}^{(m)}\right)\in \mathcal{R}^{n+1}\setminus\{0\}$, such that
    \begin{equation}\label{eqn:aConds}
       \begin{array}{lcl}
  \vert Q_0+\mathbf{P}\cdot \mathbf{f}(\x_0)\vert<q^{-t},\\[2ex]
  \left|{P}^{(d)}_i+\sum_{j=1}^m\partial_if_j(\x_0){P}^{(m)}_j\right|<q^{\frac{s-t}{2}},\,\text{    for  }i=1,\dots,d, \\[2ex] 
  \vert {P}^{(m)}_j\vert<q^s,\,\text{    for  }j=1,\dots,m.
\end{array} 
    \end{equation}  
    By the second and third inequalities in \eqref{eqn:aConds}, for every $i=1,\dots,d$,
    \begin{equation}
    \label{eqn:a_1...a_nPart}
        \left|{P}^{(d)}_i\right|\leq \max\left\{\left|{P}^{(d)}_i+\sum_{j=1}^m\partial_if_j(\x_0){P}^{(m)}_j\right|,\max_{j=1,\dots,m}\left|\partial_if_j(\x_0){P}^{(m)}_j\right|\right\}\leq M
        q^{s}.
    \end{equation}
    If $\x\in \mathfrak{M}(s,t)$, then there exists $\x_0\in \mathfrak{M}_0(s,t)$, such that $\Vert \x-\x_0\Vert<q^{-\frac{s+t}{2}}$. Considering the Taylor expansion about $\x_0$,  it follows from Lemma \ref{lem:Taylor}, \ref{IV}, \eqref{eqn:aConds}, and  \eqref{eqn:a_1...a_nPart} that
  
\begin{equation}\label{eqn:f_P_bound}
|Q_0 + \mathbf{P}\cdot \mathbf{f}(\x)|
\le \max \left\{q^{-t}, Mq^{-\frac{s+t}{2}}q^{\frac{s-t}{2}}, Mq^s\, \left(q^{-(\frac{s+t}{2})}\right)^2\right\}<\max\{1,M\}q^{-t}.
\end{equation}
    Again using \eqref{eqn:a_1...a_nPart} and the first-order Taylor expansion of $\partial_if_j$ about $\mathbf{x}_0$ \eqref{eqn:a_1...a_nPart}, we get
\begin{align}\label{eqn:gradient_bound}
\left|{P}_i^{(d)}+\sum_{j=1}^m\partial_if_j(\x){P}_j^{(m)}\right|
&\leq
\max\left\{
\left|{P}_i^{(d)}+\sum_{j=1}^m\partial_if_j(\x_0){P}_j^{(m)}\right|,
\max_{j=1,\dots,m}\left|\partial_if_j(\x)-\partial_if_j(\x_0)\right|
\cdot \left|{P}_j^{(m)}\right|
\right\}\nonumber \\
&\leq
\max\left\{q^{\frac{s-t}{2}},Mq^{-\frac{s+t}{2}}q^s\right\} \leq
\max\{1,M\}q^{\frac{s-t}{2}}.
\end{align}
Finally, combining \eqref{eqn:f_P_bound}, \eqref{eqn:gradient_bound}, \eqref{eqn:aConds}, and \eqref{eqn:a_1...a_nPart}, there exists a constant $\widetilde{C} \in \mathbb N$ that depends only on $M$, $d$, and $n$, such that the set $\mathfrak{G}_{\mathbf{f}}\left(\widetilde{C}t,\widetilde{C}\lceil(s-t)/2\rceil,\widetilde{C}s,\dots,\widetilde{C}s\right)$ contains $\mathfrak{M}(s,t)$.
 
\end{proof}
\begin{lemma}
\label{lem:M^newBND}
    Let $U\subseteqq \K^d$ be an open set and $\mathbf{f}:U\rightarrow \K^n$ be of the form $\mathbf{f}(\x)=(\x,\bm{f}(\x))$ such that it satisfies the hypothesis of Proposition \ref{prop:Key}. Then for every $\x_0\in U$ there exists a ball $B$ centered at $\x_0$ and constants $C,\alpha>0$ such that
    $$\mathcal{L}_d(\mathfrak{M}(s,t)\cap B)\leq Cq^{\frac{\alpha}{n+1}
    \left(ns-\frac{3t}{2}\right)}\mathcal{L}_d(B),\text{ whenever $t\in \N$ is large enough and }s\leq \frac{\kappa t}{n}.$$
\end{lemma}
\begin{proof}
Note that, using Lemma \ref{lem:specCont},  in the context of Theorem \ref{BKM Function Field}, up to a constant $\widetilde{C}$, we have $$r=t, s'=\lceil{(s-t)/2}\rceil,t_1=\cdots=t_n =s.$$ For $t\gg 1$, clearly $s'+\sum_{i}t_i-r-\max_{i}t_i<(n-1)s-t+\frac{s-t}{2}+1<1+ns-\frac{3t}{2}\leq 1 -\left( \kappa -\frac{3}{2}\right)t<0$, holds for all $s\in \N$ such that $s\leq \frac{\kappa t}{n}$.
 
By Theorem \ref{BKM Function Field} and Lemma \ref{lem:specCont}, for every $\x_0\in U$ there exists a ball $B$ centered at $\x_0$ and constants $\hat{C},\alpha>0$  such that 
\begin{equation*}
\begin{split}
\label{eqn:M^newMeasure}
    \mathcal{L}_d\left(\mathfrak{M}(s,t)\cap B\right)\leq \hat{C} \max\left\{q^{-t},q^{\left(\frac{2n-1}{2(n+1)}s-\frac{3t}{2(n+1)}+1\right)}\right\}^{\alpha}\mathcal{L}_d(B) &= \hat{C} q^{\alpha\left(\frac{2n-1}{2(n+1)}s-\frac{3t}{2(n+1)}+1\right)}\mathcal{L}_d(B)
 \\& \leq \left(\hat{C}q^{\alpha}\right)q^{\frac{\alpha}{n+1}
    \left(ns-\frac{3t}{2}\right)}\mathcal{L}_d(B).
\end{split}
\end{equation*}
\end{proof}
By combining Theorem \ref{thm:Covering}, Lemma \ref{lem:GenericCNT}, and Lemma \ref{lem:M^newBND} we obtain Proposition \ref{prop:Key}. We now proceed to prove the main theorems.
\section{Proofs of Theorem \ref{thm:main} and Theorem \ref{thm:DimConv}}\label{sec: proof of main theorem}
\subsection{Proof of Theorem \ref{thm:main}} To begin with, we remind the reader that the approximating function $\psi$ has the decay rate specified in \eqref{lower bound approx}.  It is enough to show that for any $ \x_0 \in U $ and any small ball $ B_0 \ni \x_0$, one has 
\[
\mathcal{L}_d\left(\left\{ \x \in B_0 : \mathbf{f}(\x) \in \mathcal{S}_n^{\boldsymbol{\Theta}}(\psi) \right\}\right) = 0.
\]
We divide the manifold into special and generic parts, i.e., $\mathfrak{M}(-\log_q \psi(q^{t}), t)$ and  $B_0 \setminus \mathfrak{M}(-\log_q \psi(q^{t}), t)$. To this end, we first observe that if $\mathbf{f}(\x) \in \mathcal{S}_n^{\boldsymbol{\Theta}}(\psi)$, then there exist infinitely many $t \in \N$ such that
\[
\left\| \mathbf{f}(\x) - \frac{\p+\boldsymbol{\Theta}}{Q} \right\| \leq \frac{\psi(q^{t})}{q^{t}} \quad \text{with} \,\, (\mathbf{P},Q) \in \mathcal{R}^{n}\times(\mathcal{R}\setminus\{0\})\,\, \text{and}\,\, |Q| = q^t.
\]Hence, for every $N \in \N$, we have
\begin{equation}
\begin{split}
\label{eqn:S_nContained}
    \left\{\x\in B_0:\mathbf{f}(\x)\in \mathcal{S}_n^{\boldsymbol{\Theta}}(\psi)\right\}\subseteqq\bigcup_{t\geq N}\underbrace{\mathfrak{M}(-\log_q \psi(q^{t}),t)\cap B_0)}_{A_t^{\boldsymbol{\Theta}}}\bigcup\\
    \bigcup_{t\geq N}\underbrace{\bigcup_{\tiny(\mathbf{P},Q)\in \mathcal{R}_{\boldsymbol{\Theta}}(B_0\setminus \mathfrak{M}(-\log_q \psi(q^{t}),t))}\left\{\x\in B_0:\left\Vert \x-\frac{\pi(\mathbf{P}+\boldsymbol{\Theta})}{Q}\right\Vert\leq \frac{\psi(q^t)}{q^t}\right\}}_{B_t^{\boldsymbol{\Theta}}},
\end{split}
\end{equation}
where $\pi: \K^n\to\K^d,$ is the projection into the first $d$ coordinates.  We invoke Proposition \ref{prop:counting generic} to obtain
 \begin{equation}
\begin{split}
\label{eqn:B_tMeasBnd}
    \mathcal{L}_d(B_t^{\boldsymbol{\Theta}})\leq q^{t(d+1)}\psi(q^t)^m\frac{\psi(q^t)^d}{q^{td}}=q^t\psi(q^t)^n,~\forall t\gg1.
\end{split}
\end{equation}
Lemma \ref{lem:M^newBND} will be applied to estimate the measure of $A_t^{\boldsymbol{\Theta}}$. 
Since \eqref{lower bound approx} is equivalent to \[-\log_q\psi(q^t)\leq \frac{\kappa t}{n}, \text{ for all }t\in \N,\] Lemma \ref{lem:M^newBND} yields an $\alpha>0$ such that, for all sufficiently large $t\in \N$, 
\begin{equation}
\label{eqn:A_tMeaBND}
\begin{split}
\mathcal{L}_d\left(A_t^{\boldsymbol{\Theta}}\right)\leq  Cq^{\frac{\alpha}{n+1}
    \left(-n\log_q\psi(q^t)-\frac{3t}{2}\right)}\mathcal{L}_d(B)
\leq C q^{\frac{\alpha }{(n+1)}\left(\kappa -\frac{3}{2}\right)t}\mathcal{L}_d(B_0). 
\end{split}
\end{equation}

From  \eqref{eqn:B_tMeasBnd}, \eqref{eqn:A_tMeaBND} and \ref{rem-conv sum}, we conclude  that
$$\sum_{t\geq N}\mathcal{L}_d(B_t^{\boldsymbol{\Theta}})+\sum_{t\geq N}\mathcal{L}_d\left(A_t^{\boldsymbol{\Theta}}\right)<\infty, \text{ when \(N\) is large enough}.$$ 
$\mathcal{L}_d(\mathbf{f}^{-1} (\mathcal{S}_n^{\boldsymbol{\Theta}}(\psi)))=0$ is now immediate from this in view of the Borel-Cantelli lemma. 
\subsection{Proof of Theorem \ref{thm:DimConv}}
\label{subsec:Dimension}
Take $s=-\log_q\psi(q^t)>0$. By Proposition \ref{prop:Key} and Theorem \ref{thm:Covering}, one can cover $A_t^{\boldsymbol{\Theta}}=\mathfrak{M}(s,t)\cap B_0$ with disjoint balls of radius $q^{-\frac{s+t}{2}}$. When  $t\in \N$ is large enough, one can cover it with at most $ q^{\frac{d}{2}(s+t)}\mathcal{L}_d(A_t^{\boldsymbol{\Theta}})$ such balls. Thus, by Proposition \ref{prop:Key}, there exists $N_1 \in \N$ such that for every $t\geq N_1$, we have 
\begin{equation}
    \mathcal{H}^{\sigma}(A_t^{\boldsymbol{\Theta}})\ll q^{-\frac{s+t}{2}(\sigma-d)}\mathcal{L}_d(A_t^{\boldsymbol{\Theta}})\ll q^{-\frac{s+t}{2}(\sigma-d)}q^{{\frac{\alpha}{n+1}}(ns-\frac{3t}{2})}.
\end{equation}
\noindent Hence, by plugging in $s=-\log_q\psi(q^t)$, one obtains from \eqref{eq:hausdim} and \ref{rem-conv sum} that
\begin{equation}\label{N_1}
    \sum_{t\geq N_1}\mathcal{H}^{\sigma}(A_t^{\boldsymbol{\Theta}})\ll \sum_{t\geq N_1}\left(\frac{\psi(q^t)}{q^t}\right)^{\frac{\sigma-d}{2}}\left(\psi(q^t)^nq^{\frac{3t}{2}}\right)^{-\frac{\alpha}{n+1}}<\infty. 
\end{equation}
For the set $ B_{t}^{\boldsymbol{\Theta}} $, we employ the counting estimate~\eqref{eqn:GenCnt} along with a covering by balls of radius $ r = q^{-t}\psi(q^{t}) $. Hence, there exists $N_2\in \N$ such that
\begin{equation}\label{N_2}
   \sum_{t\geq N_2}\mathcal{H}^{\sigma}(B_t^{\boldsymbol{\Theta}})\ll \sum_{t\geq N_2}q^{(d+1)t}q^{-ms}r^{\sigma}
    =\sum_{t\geq N_2}\left({\frac{\psi(q^t)}{q^t}}\right)^{\sigma+m}q^{(n+1)t}<\infty,   
\end{equation}
owing to \eqref{eq:hausdim} and \ref{rem-conv sum}. It follows at once from \eqref{eqn:S_nContained}, \eqref{N_1} and \eqref{N_2} that $\mathcal{H}^{\sigma}(\mathbf{f}^{-1}(\mathcal{S}_n^{ \boldsymbol{\Theta}}(\psi)))=0$, due to the Borel-Cantelli lemma. 
\section{Appendix - A Result in Geometry of Numbers for Discrete Subgroups}\label{sec:minima}
We consider an extension of the function field $\K$ obtained by adjoining the power $T^{\frac{1}{\ell}}$. As mentioned earlier, we denote this extension by $\K_{\ell}$. Similarly, we denote the lattice $\F_q[T^{1/\ell}]\subseteqq \K_{\ell}$ as $\widetilde{\mathcal{R}} \df\mathbb{F}_q[T^{\frac{1}{\ell}}]$. We now prove the following proposition, which establishes a relationship between the first minimum of the discrete subgroup $gu \mathcal{R}^{m+n}\subseteqq\K_{\ell}^{m+n}$ and that of the lattice $gu \widetilde{\mathcal{R}}^{m+n}\subseteqq\K_{\ell}^{m+n}$, where $g$ is some diagonal matrix in $\operatorname{SL}(m+n,\K_{\ell})$ and $u$ is some unipotent matrix in $\operatorname{SL}(m+n,\K)$.
\begin{proposition}\label{prop:minima comparison}
 Let $g\in \operatorname{SL}\left(m+n,\K_{\ell}\right)$ be a diagonal matrix, $\boldsymbol{\alpha}\in M_{m\times n}(\K)$ and let
    $$u=\begin{pmatrix}
        \mathrm{I}_m&\boldsymbol{\alpha}\\
        0&I_n
    \end{pmatrix}\in\operatorname{SL}\left(m+n,\K\right).$$
    Then, $\lambda_1(gu\widetilde{\mathcal{R}}^{m+n})=\lambda_1(gu\mathcal{R}^{m+n})$.
\end{proposition}
\begin{proof}
 Let $(\mathbf{P},\mathbf{Q})=(P_1,\dots,P_m,Q_1,\dots,Q_n)^T\in \widetilde{\mathcal{R}}^{m+n}$ be a vector. Explicitly write $P_i=\sum_{k=0}^{\ell-1}T^{\frac{k}{\ell}}P_{i,k}$, where $P_{i,k}\in \mathcal{R}$, and $Q_j=\sum_{k=0}^{\ell-1}T^{\frac{k}{\ell}}Q_{j,k}$, where $Q_{j,k}\in \mathcal{R}$. Since every $\alpha\in \K_{\ell}$ can be written as $\alpha=T^{\deg(\alpha)}\cdot \frac{\alpha}{T^{\deg(\alpha)}}$, where $\deg(\alpha)\in \frac{1}{\ell}\mathbb{Z}$ and $\left|\frac{\alpha}{T^{\deg(\alpha)}}\right|=1$, there exist $r_1,\dots, r_{m+n}\in \mathbb{Z}$ and $\mathbf{o}=\operatorname{diag}\{o_1,\dots,o_{m+n}\}$, with $\vert o_j\vert=1$ for every $j=1,\dots,m+n$, such that
    $$g=\begin{pmatrix}
        T^{\frac{r_1}{\ell}}&&&\\
        &\ddots&&\\
        &&T^{\frac{r_{m+n}}{\ell}}&
    \end{pmatrix}\mathbf{o}.$$
Since $\mathbf{o}$ does not affect the norm, we may assume that $\mathbf{o}=\mathrm{Id}$. Therefore, 
    \begin{equation}
    \label{eqn:auVec}
        gu\begin{pmatrix}
            \mathbf{P}\\
            \mathbf{Q}
        \end{pmatrix}=\begin{pmatrix}
            P_1+\sum_{j=1}^n\alpha_{1,j}Q_{j}\\
            \vdots\\
            P_m+\sum_{j=1}^n\alpha_{m,j}Q_{j}\\
            Q_{1}\\
            \vdots\\
            Q_{n}
        \end{pmatrix}=\sum_{k=0}^{\ell-1}T^{\frac{k}{\ell}}\begin{pmatrix}
            P_{1,k}+\sum_{j=1}^n\alpha_{1,j}Q_{j,k}\\
            \vdots\\
            P_{m,k}+\sum_{j=1}^n\alpha_{m,j}Q_{j,k}\\
            Q_{1,k}\\
            \vdots\\
            Q_{n,k}
        \end{pmatrix}.
    \end{equation}
    Note that the norm of the $k$-th summand on the right-hand side of \eqref{eqn:auVec} is in $q^{\mathbb{Z}+\frac{k}{\ell}}$. Hence, every summand of each entry of \eqref{eqn:auVec} has a different absolute value or is equal to zero. Hence, 
    \begin{equation}
    \label{eqn:gu(P,Q)Norm}
        \left\Vert gu\begin{pmatrix}
            \mathbf{P}\\
            \mathbf{Q}
        \end{pmatrix}\right\Vert=\max_{k=0,\dots,\ell-1}q^{\frac{k}{\ell}}\left\Vert \begin{pmatrix}
            P_{1,k}+\sum_{j=1}^n\alpha_{1,j}Q_{j,k}\\
            \vdots\\
            P_{m,k}+\sum_{j=1}^n\alpha_{m,j}Q_{j,k}\\
            Q_{1,k}\\
            \vdots\\
            Q_{n,k}
        \end{pmatrix}\right\Vert.
    \end{equation}
    For $k=0,1,\dots,\ell-1$, let $\mathbf{P}^{(k)}=(P_{1,k},\dots,P_{m,k})^T$ and $\mathbf{Q}^{(k)}=(Q_{1,k},\dots,Q_{n,k})$. Let $k$ be such that 
    $$\left\Vert gu\begin{pmatrix}
        \mathbf{P}^{(k)}\\
        \mathbf{Q}^{(k)}
    \end{pmatrix}\right\Vert=\min\left\{\left\Vert
        gu\begin{pmatrix}
            \mathbf{P}^{(i)}\\
            \mathbf{Q}^{(i)}
        \end{pmatrix}
    \right\Vert:gu\begin{pmatrix}
        \mathbf{P}^{(i)}\\
        \mathbf{Q}^{(i)}
    \end{pmatrix}\neq 0\right\}.$$ Then, by \eqref{eqn:gu(P,Q)Norm}, 
    \begin{equation}\begin{split}\left\Vert gu\begin{pmatrix}
        \mathbf{P}\\
        \mathbf{Q}
    \end{pmatrix}\right\Vert=\max_{i=0,\dots,\ell-1}q^{\frac{i}{\ell}}\left\Vert gu\begin{pmatrix}
        \mathbf{P}^{(i)}\\
        \mathbf{Q}^{(i)}
    \end{pmatrix}\right\Vert\geq \left\Vert gu\begin{pmatrix}
        \mathbf{P}^{(k)}\\
        \mathbf{Q}^{(k)}
    \end{pmatrix}\right\Vert\end{split}\end{equation}
    Furthermore $\begin{pmatrix}
        \mathbf{P}^{(k)}\\
        \mathbf{Q}^{(k)}
    \end{pmatrix}\in \mathcal{R}^{m+n}$. As a consequence, there exists a vector $\begin{pmatrix}
        \mathbf{P}\\
        \mathbf{Q}
    \end{pmatrix}\in \mathcal{R}^{m+n}$ such that 
    $$\lambda_1\left(gu\widetilde{\mathcal{R}}^{m+n}\right)=\left\Vert gu\begin{pmatrix}
        \mathbf{P}\\
        \mathbf{Q}
    \end{pmatrix}\right\Vert.$$
    Therefore, $\lambda_1\left(gu\widetilde{\mathcal{R}}^{m+n}\right)\geq \lambda_1\left(gu\mathcal{R}^{m+n}\right)$. On the other hand, $gu\mathcal{R}^{m+n}\subseteqq gu\widetilde{\mathcal{R}}^{m+n}$, and thus, $\lambda_1\left(gu\widetilde{\mathcal{R}}^{m+n}\right)=\lambda_1\left(gu\mathcal{R}^{m+n}\right).$
\end{proof}

\begin{remark}
By Proposition~\ref{prop:minima comparison}, the set introduced in equation~\eqref{Mnew} can equivalently be written as
\[
\mathfrak{M}_0(s,t) 
= \left\{ 
\x \in U : 
\lambda_1\!\left(g_{s,t}^{\star} u_1^{\star}(\x) \mathcal{R}^{n+1}\right) 
= 
\lambda_1\!\left(g_{s,t}^{\star} u_1^{\star}(\x) \widetilde{\mathcal{R}}^{n+1}\right) 
< 
q^{-\frac{(d+2)t - 2ns}{2(n+1)}}
\right\}.
\]
This formulation offers an alternative characterization, and, most importantly, note that $g_{s,t}^{\star} u_1^{\star}(\x) \widetilde{\mathcal{R}}^{n+1}$ is a lattice in $\K_{\ell}^{n+1}$. Hence, one can use the Duality theorem over the extended field $\K_{\ell}$ for the aforementioned lattice and have an alternative approach to provide a counting estimate for the generic part of the manifold~$\mathscr{M}$.

\end{remark}
Proposition \ref{prop:minima comparison} gives rise to several questions about the comparison between successive minima of discrete subgroups in extension fields and lattices in extension fields. Hence, we conclude this section with the following questions.
\begin{question}
    Let $\ell \geq 2$ and $n\in \N$. Also let $g\in \operatorname{SL}(n,\K_{\ell})$. 
    \begin{enumerate}
        \item When do we have $\lambda_1(g\mathcal{R}^n)=\lambda_1(g\widetilde{\mathcal{R}}^n)$?
        \item Given any $2\leq i \leq n$, when do we have $\lambda_i(g\mathcal{R}^n)=\lambda_i(g\widetilde{\mathcal{R}}^n)$?
        \item Given $2\leq i\leq n$, when do we have $\lambda_j(g\mathcal{R}^n)=\lambda_j(g\widetilde{\mathcal{R}}^n)$ for every $j=1,\dots,i$? 
    \end{enumerate}
\end{question}
\bibliography{mybib}
\bibliographystyle{amsalpha}
\end{document}